\documentclass[a4paper,11pt,reqno]{amsart}

\usepackage[utf8]{inputenc}
\usepackage[T1]{fontenc}
\usepackage{lmodern}
\usepackage[english]{babel}
\usepackage{microtype}

\usepackage{amsmath,amssymb,amsfonts,amsthm,comment}
\usepackage{mathtools,accents}
\usepackage{mathrsfs}
\usepackage{aliascnt}
\usepackage{braket}
\usepackage{bm}
\usepackage{esint}
\usepackage{amsfonts}
\usepackage{csquotes}

\usepackage[a4paper,margin=2.7cm]{geometry}
\usepackage[citecolor=blue,colorlinks]{hyperref}
\usepackage[english,capitalize]{cleveref}

\usepackage{enumerate}
\usepackage{xcolor}

\usepackage{imakeidx} 
\usepackage{xparse} 
\usepackage{mathtools}
\usepackage{float}
\usepackage{nicefrac}
\usepackage{graphicx}
\usepackage{enumerate}

\usepackage[maxbibnames=99,backend=biber, sorting=nyt, url=false]{biblatex}
\renewbibmacro{in:}{%
  \ifentrytype{article}
    {}
    {\bibstring{in}%
     \printunit{\intitlepunct}}}

\makeatletter
\g@addto@macro\@floatboxreset\centering
\makeatother

\makeatletter
\def\newaliasedtheorem#1[#2]#3{
  \newaliascnt{#1@alt}{#2}
  \newtheorem{#1}[#1@alt]{#3}
  \expandafter\newcommand\csname #1@altname\endcsname{#3}
}
\makeatother

\numberwithin{equation}{section}

\newtheoremstyle{slanted}{\topsep}{\topsep}{\slshape}{}{\bfseries}{.}{.5em}{}

\theoremstyle{plain}
\newtheorem{theorem}{Theorem}[section]
\newaliasedtheorem{proposition}[theorem]{Proposition}
\newaliasedtheorem{question}[theorem]{Question}
\newaliasedtheorem{lemma}[theorem]{Lemma}
\newaliasedtheorem{corollary}[theorem]{Corollary}
\newaliasedtheorem{counterexample}[theorem]{Counterexample}

\theoremstyle{definition}
\newaliasedtheorem{definition}[theorem]{Definition}
\newaliasedtheorem{example}[theorem]{Example}
\newaliasedtheorem{conjecture}[theorem]{Conjecture}

\theoremstyle{remark}
\newaliasedtheorem{remark}[theorem]{Remark}

\newcommand{\setN}{\mathbb{N}}
\newcommand{\N}{\mathbb{N}}

\newcommand{\setR}{\mathbb{R}}
\newcommand{\R}{\mathbb{R}}

\newcommand{\mm}{\mathfrak m}
\newcommand{\X}{{\rm X}}
\newcommand{\lip}{{\rm lip \,}}
\newcommand{\nchi}{{\raise.3ex\hbox{\(\chi\)}}}

\newcommand{\eps}{\varepsilon}

\let\phi\varphi

\newcommand{\abs}[1]{\left\lvert#1\right\rvert}
\newcommand{\norm}[1]{\left\lVert#1\right\rVert}

\DeclareMathOperator{\sn}{sn}

\newcommand{\di}{\mathop{}\!\mathrm{d}}

\newcommand{\res}{\mathop{\hbox{\vrule height 7pt width .5pt depth 0pt
\vrule height .5pt width 6pt depth 0pt}}\nolimits}

\newcommand{\Ch}{{\sf Ch}}

\newcommand{\st}{\ensuremath{\ :\ }} 
\newcommand{\eqdef}{\ensuremath{\vcentcolon=}}

\DeclareMathOperator{\Cb}{C_b}

\newcommand{\haus}{\mathscr{H}}

\newcommand{\ppi}{{\boldsymbol{\pi}}}

\newcommand{\dist}{\mathsf{d}}

\newcommand{\meas}{\mathfrak{m}}

\newcommand{\diam}{\mathrm{diam}}
\DeclareMathOperator{\CD}{CD}
\DeclareMathOperator{\RCD}{RCD}

\DeclareMathOperator{\Per}{Per}

\newfont{\tmpf}{cmsy10 scaled 2500}

\newcommand{\de}{\ensuremath{\,\mathrm d}} 


\begin{document}

\title[Isoperimetry under lower Ricci bounds]{Sharp isoperimetric comparison on non-collapsed spaces with lower Ricci bounds\\[2ex] Comparaison isop\'erim\'etrique optimale pour les espaces non effondr\'es \`a courbure de Ricci minor\'ee
}

\author[G. Antonelli]{Gioacchino Antonelli}\address{Courant Institute Of Mathematical Sciences (NYU), 251 Mercer Street, 10012, New York, USA}\email{ga2434@nyu.edu}

\author[E. Pasqualetto]{Enrico Pasqualetto}\address{Scuola Normale Superiore, Piazza dei Cavalieri, 7, 56126 Pisa, Italy.}\email{enrico.e.pasqualetto@jyu.fi}

\author[M. Pozzetta]{Marco Pozzetta}\address{Dipartimento di Matematica e Applicazioni, Universit\`a di Napoli Federico II, Via Cintia, Monte S. Angelo, 80126 Napoli, Italy.}\email{marco.pozzetta@unina.it}

\author[D. Semola]{Daniele Semola}\address{Mathematical Institute, University of Oxford, Radcliffe Observatory, Andrew Wiles Building, Woodstock Rd, Oxford OX2 6GG}\email{daniele.semola.math@gmail.com}

\maketitle

\begin{abstract}
  This paper studies sharp
  isoperimetric comparison theorems and sharp dimensional concavity properties of the isoperimetric profile for non-smooth spaces with lower Ricci curvature bounds, the so-called $N$-dimensional $\RCD(K,N)$ spaces. 
  
  The absence of most of the classical tools of geometric measure theory and the possible non-existence of isoperimetric regions on non-compact spaces are handled via an original argument to estimate first and second variation of the area for isoperimetric sets, avoiding any regularity theory, in combination with
  an asymptotic mass decomposition result of perimeter-minimizing sequences.
 
  Most of our statements are new even for smooth, non-compact manifolds with lower Ricci curvature bounds and for Alexandrov spaces with lower sectional curvature bounds. They generalize several results known for compact manifolds, non-compact manifolds with uniformly bounded geometry at infinity, and Euclidean convex bodies. 
\end{abstract}

\renewcommand{\abstractname}{Résumé}
\begin{abstract}
Cet article étudie les théorèmes de comparaison isopérimétrique et les propriétés de concavité du profil isopérimétrique pour les espaces non lisses avec \`a courbure de Ricci minor\'ee: les espaces $\RCD(K,N)$ de dimension $N$.

L'absence de la plupart des outils classiques de théorie géométrique de la mesure et la non-existence possible de régions isopérimétriques dans les espaces non compacts sont traitées au moyen d'un argument original pour estimer la première et la deuxième variation de l'aire pour les ensembles isopérimétriques, en évitant la théorie de régularité. Cet argument est combiné avec avec un résultat de décomposition asymptotique de masse pour les suites minimisant le périmètre.

La plupart de nos énoncés sont nouveaux même pour les variétés lisses non compactes avec \`a courbure de Ricci minor\'ee, et pour les espaces d'Alexandrov \`a courbure sectionnelle minor\'ee. Ils généralisent plusieurs résultats connus pour les variétés compactes, les variétés non compactes avec \`a géométrie uniformément bornée à l'infini, et les corps convexes euclidiens.
\end{abstract}

\setcounter{tocdepth}{1}
\tableofcontents

\section{Introduction}

\subsection*{Isoperimetry and lower Ricci curvature bounds}

There is a celebrated connection between Ricci curvature and the isoperimetric problem in geometric analysis, going back at least to the L\'evy--Gromov inequality \cite[Appendix C]{Gromovmetric}. The primary goal of this paper is to extend several results about the isoperimetric problem on compact Riemannian manifolds with lower Ricci curvature bounds to non-compact Riemannian manifolds and non-smooth spaces with lower Ricci curvature bounds. In order to deal with the possible non-existence of isoperimetric regions and with the lack of regularity we develop a series of new tools with respect to the classical literature. Non-smooth spaces enter into play naturally when dealing with smooth non-compact Riemannian manifolds, via the analysis of their pointed limits at infinity.
\smallskip

We consider the setting of $N$-dimensional $\RCD(K,N)$ metric measure spaces $(X,\dist,\haus^N)$, for finite $N\in [1,\infty)$ and $K\in \setR$, see \cite{DePhilippisGigli18,Kitabeppu17} after \cite{Sturm1,Sturm2,LottVillani,AmbrosioGigliSavare14,Gigli12,AmbrosioGigliMondinoRajala15,ErbarKuwadaSturm15,AmbrosioMondinoSavare15,CavallettiMilmanCD}. Here $K\in \setR$ plays the role of (synthetic) lower bound on the Ricci curvature, $N\in [1,\infty)$ plays the role of (synthetic) upper bound on the dimension and $\haus^N$ indicates the $N$-dimensional Hausdorff measure. This class includes (convex subsets of) smooth Riemannian manifolds with lower Ricci curvature bounds endowed with their volume measure, their noncollapsing measured Gromov--Hausdorff limits \cite{ChCo1}, and finite dimensional Alexandrov spaces with sectional curvature lower bounds \cite{BuragoGromovPerelman,PetruninLSV}.
\smallskip

We shall rely on the theory of sets of finite perimeter in $\RCD(K,N)$ spaces, as developed in \cite{Ambrosio02,AmbrosioBrueSemola19,BruePasqualettoSemola,BPSGaussGreen}. For the sake of this introduction we just remark that it is fully consistent with the Euclidean and Riemannian ones. In particular, (reduced) boundaries of sets of finite perimeter are rectifiable, the perimeter coincides with the restriction of the $(N-1)$ --dimensional Hausdorff measure to the (reduced) boundary and it does not charge the boundary of the ambient space.\\
Given an $\RCD(K,N)$ metric measure space $(X,\dist,\haus^N)$ such that $\haus^N(B_1(x))\ge v_0$ for any $x\in X$ for some $v_0>0$, we introduce the isoperimetric profile $I_X:[0,\haus^N(X))\to [0,\infty)$ by
\begin{equation}\label{eq:isoIntro}
    I_X(v):=\inf\left\{\Per(E)\, :\, E\subset X\, ,\, \, \haus^N(E)=v\right\}\, ,
\end{equation}
where we drop the subscript $X$ when there is no risk of confusion.
When $E\subset X$ attains the infimum in \eqref{eq:isoIntro} for $v=\haus^N(E)$, we call it an isoperimetric region.
In this setting we obtain:
\begin{itemize}
    \item sharp second order differential inequalities for the isoperimetric profile, corresponding to equalities on the model spaces with constant sectional curvature. These inequalities are new even in the case of non-compact Riemannian manifolds and use in a crucial way the non-smooth approach.  The proof bypasses the possible non-existence of isoperimetric regions on the space, that is classically used for such arguments, employing a concentration-compactness argument; 
    
    \item a sharp Laplacian comparison theorem for the distance function from $\partial E$, which is a fundamental tool to prove the above items since it corresponds to the bounds usually obtained via first and second variation of the area in this low regularity setting;
    \item Gromov--Hausdorff stability and perimeters' convergence of isoperimetric regions along non-collapsing sequences of $N$-dimensional $\RCD(K,N)$ spaces. In order to prove these statements we deduce uniform regularity estimates for isoperimetric sets from uniform concavity estimates of the isoperimetric profiles.
\end{itemize}

Many of the above results are new even for smooth, non-compact manifolds with lower Ricci curvature bounds and for Alexandrov spaces with lower sectional curvature bounds. They answer several open questions in \cite{Bayle03, Milmanconvexity, LedouxCtoIso, NardulliOsorio20, BaloghKristaly}.\\
We expect the techniques developed in this paper to have a broad range of applications in geometric analysis under lower curvature bounds. For instance, in the study of more general geometric variational problems, the isoperimetric problem on weighted Riemannian manifolds with lower bounds on the Bakry-\'{E}mery curvature tensor, other geometric and functional inequalities.


\subsection*{Main results}

On model spaces with constant sectional curvature $K/(N-1)\in\setR$ and dimension $N\ge 2$ the isoperimetric profile $I_{K,N}$ solves the following second order differential equation on its domain:
\begin{equation}\label{BPintro}
    -I_{K,N}''I_{K,N}=K+\frac{\left(I'_{K,N}\right)^2}{N-1}\, .
\end{equation}
Equivalently, setting $\psi_{K,N}:=I_{K,N}^{\frac{N}{N-1}}$, we have
    \begin{equation}\label{Bayleintro}
    -\psi_{K,N}''= \frac{KN}{N-1}\psi_{K,N}^{\frac{2-N}{N}}\, .
    \end{equation}
Combining the existence of isoperimetric regions for any volume, the regularity theory in geometric measure theory, and the second variation of the area \eqref{eq:secondvariation}, in \cite{BavardPansu86,Bayle03,Bayle04,MorganJohnson00,NiWangiso} it was proved that the isoperimetric profile of a smooth, compact, $N$-dimensional Riemannian manifold with $\mathrm{Ric}\ge K$ verifies the inequality $\ge$ in \eqref{BPintro} and \eqref{Bayleintro} in a weak sense. 

Here we obtain the following extension to the setting of $\RCD(K,N)$ metric measure spaces $(X,\dist,\haus^N)$ with a uniform lower bound on the volume of unit balls, without any assumption on the existence of isoperimetric regions. We stress again that the classical argument to show \autoref{thm:BavardPansuIntro} in the compact setting uses in a crucial way the existence of isoperimetric regions for every volume, that we do not have at disposal in the present setting.

\begin{theorem}[cf.\ with \autoref{thm:BavardPansu}]\label{thm:BavardPansuIntro}
Let $(X,\dist,\haus^N)$ be an $\RCD(K,N)$ space. Assume that there exists $v_0>0$ such that $\haus^N(B_1(x))\geq v_0$ for every $x\in X$. 

Let $I:(0,\haus^N(X))\to (0,\infty)$ be the isoperimetric profile of $X$. Then:
\begin{enumerate}
    \item the inequality
 \[
    -I''I\geq K+\frac{(I')^2}{N-1}\,\quad\text{holds in the viscosity sense on $(0,\haus^N(X))$}\,,
 \]
    \item  if $\psi:=I^{\frac{N}{N-1}}$ then 
\[
    -\psi''\geq \frac{KN}{N-1}\psi^{\frac{2-N}{N}}\,\quad\text{holds in the viscosity sense on $(0,\haus^N(X))$}\, .
\]
\end{enumerate}
\end{theorem}

{
In particular, the above holds for non-compact smooth Riemannian manifolds with Ricci curvature bounded from below and volume of unit balls uniformly bounded away from zero. In the smooth noncompact setting, \autoref{thm:BavardPansuIntro} was previously known only under the additional assumption of existence of isoperimetric sets, or under strong conditions on the asymptotic geometry, see \cite{MondinoNardulli16}. To the best of our knowledge, this is the first application of the theory of $\RCD$ spaces to prove a new sharp geometric inequality on smooth Riemannian manifolds.}

%

The proof of \autoref{thm:BavardPansuIntro} combines the generalized existence of isoperimetric regions (cf.\ with \autoref{thm:MassDecompositionINTRO}), the interpretation of the differential inequalities in the viscosity sense and the forthcoming Laplacian comparison \autoref{thm:Isoperimetriciintro} to estimate first and second variation of the area via equidistants in the non-smooth setting. 
\medskip


The main new tool that we develop to prove \autoref{thm:BavardPansuIntro} is a sharp bound on the Laplacian of the signed distance function from isoperimetric regions inside $\RCD(K,N)$ metric measure spaces $(X,\dist,\haus^N)$. 

If $E$ is a (smooth) isoperimetric region inside a smooth Riemannian manifold $(M^N,g)$ the first variation formula implies that the mean curvature $H$ of its boundary $\partial E$ is constant. Moreover, if $t\mapsto E_t$ denotes the parallel deformation of $E$ via equidistant sets, $\nu$ and $\mathrm{II}$ denote a choice of the unit normal to $\partial E$ and its second fundamental form, respectively, and $\mathrm{Ric}(\nu,\nu)$ indicates the Ricci curvature of $M$ in the direction of $\nu$, then the second variation formula yields that 
\begin{equation}\label{eq:secondvariation}
    \frac{\di^2}{\di t^2}\lvert_{t=0}\Per(E_t)=\int_{\partial E}\left(H^2-||\mathrm{II}||^2-\mathrm{Ric}(\nu,\nu)\right)\di \Per\, ,
\end{equation}
where we denoted by $\Per$ the perimeter, which coincides with the Riemannian surface measure for sufficiently regular sets.\\
If we further assume that $\mathrm{Ric}\ge Kg$, for some $K\in \setR$, then the Cauchy--Schwarz inequality applied to the eigenvalues of $\mathrm{II}$ yields
\begin{equation}\label{eq:secondvariationest}
     \frac{\di^2}{\di t^2}\Big|_{t=0}\Per(E_t)\le \left(\frac{N-2}{N-1}H^2-K\right)\Per(E)\, .
\end{equation}
One of the main technical achievements of this work is to develop a counterpart of \eqref{eq:secondvariationest} for isoperimetric regions in non-smooth spaces with lower Ricci curvature bounds. Our argument departs from the classical literature, following the general scheme outlined above, and it requires a different, global rather than infinitesimal, perspective.

Let us introduce the comparison functions
\begin{equation}
    s_{k,\lambda}(r):=\cos_k(r)-\lambda\sin_k(r)\, ,
\end{equation}
where 
\begin{equation}
   \cos_k''+k\cos_k=0\, ,\quad \cos_k(0)=1\, ,\quad \cos_k'(0)=0\, , 
\end{equation}
and 
\begin{equation}
    \sin_k''+k\sin_k=0\, ,\quad\sin_k(0)=0\, ,\quad \sin_k'(0)=1\, .
\end{equation}

\begin{theorem}[cf. with \autoref{thm:Isoperimetrici}]\label{thm:Isoperimetriciintro}
Let $(X,\dist,\haus^N)$ be an $\RCD(K,N)$ metric measure space for some $K\in\setR$ and $N\ge 2$ and let $E\subset X$ be an isoperimetric region.
Then, denoting by $f$ the signed distance function from $\overline{E}$, and $N':=N-1$, there exists $c\in\setR$ such that
\begin{equation}\label{eq:sharplaplacianintro}
\boldsymbol{\Delta} f\ge  -(N-1)\frac{s'_{K/N',c/N'}\circ \left(-f\right)}{s_{K/N',c/N'}\circ \left(-f\right)}      \quad\text{on $E$, and }\quad
\boldsymbol{\Delta} f\le (N-1)\frac{s'_{K/N',-c/N'}\circ f}{s_{K/N',-c/N'}\circ f}    \quad\text{on $X\setminus \overline{E}$}\, .
\end{equation}
\end{theorem}

The bounds in \eqref{eq:sharplaplacianintro} are understood in the sense of distributions, and we always consider open representatives for isoperimetric regions (see \autoref{thm:RegularityIsoperimetricSets}). They are sharp, since equalities are attained in the model spaces with constant sectional curvature.

Notice that the distance function might not be globally smooth even when $(X,\dist)$ is isometric to a smooth Riemannian manifold and $E\subset X$ has smooth boundary, in which case \eqref{eq:sharplaplacianintro} is equivalent to the requirement that $\partial E$ has constant mean curvature equal to $c$. 
We will indicate any $c\in\setR$ such that \eqref{eq:sharplaplacianintro} holds as a \emph{mean curvature barrier} for $E$.

\subsection*{Consequences and strategies of the proofs}
Several consequences of \autoref{thm:BavardPansuIntro} are investigated in the rest of the paper:
\begin{itemize}
    \item uniform semi-concavity and Lipschitz properties of the isoperimetric profile in a fixed range of volumes, only depending on the lower Ricci curvature bound, the dimension and a lower bound on the volume of unit balls, see \autoref{prop:SharpenedConcavityAndLimit}, \autoref{cor:UniformPositivityProfile}, and \autoref{cor:UniformLipschitzProfile};
    \item the existence of the limit $\lim_{v\to 0}I(v)/v^{\frac{N-1}{N}}\in (0,N\omega_N^{\frac{1}{N}}]$ on any $\RCD(K,N)$ space $(X,\dist,\haus^N)$ with volume of unit balls uniformly bounded from below, see \autoref{prop:SharpenedConcavityAndLimit} and \autoref{rem:AsintoticaProfileSmallVolumes};
    \item the strict subadditivity of the isoperimetric profile for small volumes (only depending on $K$, $N$ and the uniform lower bound on the volume of unit balls), see \autoref{cor:EstimateDerivativeProfileAndSubadditivity}. This implies in turn that isoperimetric regions with small volume are connected, see \autoref{cor:GHRnRegioniConnesse}. Moreover, in the asymptotic mass decomposition, minimizing sequences for small volumes do not split: either they converge to an isoperimetric region, or they drift off to exactly one isoperimetric region in a pointed limit at infinity, see \autoref{lem:IsoperimetricAtFiniteOrInfinite}. All the previous conclusions hold for every volume when $K=0$;
    \item uniform, scale invariant diameter estimates for isoperimetric regions of small volume, without further assumptions, and for any volume when $K=0$ and $(X,\dist,\haus^N)$ has Euclidean volume growth, see \autoref{lem:BoundDiam};
    \item uniform density estimates and uniform almost minimality properties for isoperimetric sets, see \autoref{cor:UniformRegularityIsop}. They allow to bootstrap $L^1$-convergence to Gromov--Hausdorff convergence and convergence of the perimeters for sequences of isoperimetric sets and to prove the stability of mean curvature barriers obtained with \autoref{thm:Isoperimetriciintro}, see \autoref{thm:ConvergenceBarriersStability}. We remark that uniform almost minimality properties are fundamental in several circumstances in geometric measure theory, see for instance \cite{Whiteminimax}, and that the classical strategies to achieve them break in the present setting, as they heavily rely on smoothness.\\
\end{itemize}

Let us further comment on \autoref{thm:Isoperimetriciintro}.
When $\partial E$ is a smooth constant mean curvature hypersurface in a smooth Riemannian manifold, the Laplacian of the signed distance function from $E$ equals the mean curvature of $\partial E$ along $\partial E$. Then \eqref{eq:sharplaplacianintro} can be proved with a classical computation using Jacobi fields and one dimensional comparison for Riccati equations away from the cut locus, finding its roots in \cite{Wuelementary,Calabi58}. The singular contribution coming from the cut locus has the right sign, in great generality.
\smallskip

The original proof of the L\'evy--Gromov isoperimetric inequality \cite[Appendix C]{Gromovmetric} builds on a variant of \eqref{eq:sharplaplacianintro}. The key additional difficulty with respect to smooth constant mean curvature hypersurfaces is that boundaries of isoperimetric regions might be non-smooth when $N\ge 8$ and it is handled relying on a deep regularity theorem in geometric measure theory \cite{Almgrenreg}.
This strategy seems out of reach in the setting of $\RCD$ spaces.\\ 
Our proof of \autoref{thm:Isoperimetriciintro} partially avoids the regularity theory even on smooth Riemannian manifolds. It is inspired by \cite{CaffarelliCordoba93,Cabre97,Petruninharmonic} and the recent study of perimeter minimizing sets in \cite{MoS21}.\\ 
We start proving adimensional versions of \eqref{eq:sharplaplacianintro}, corresponding to the limit of \eqref{eq:sharplaplacianintro} as $N\to\infty$. Exploiting the equivalence between distributional and viscosity bounds on the Laplacian from \cite{MoS21}, we prove that if the bounds fail there exists a volume fixing perturbation of $E$ with strictly smaller perimeter, a contradiction with the isoperimetric condition. The perturbations are built by sliding simultaneously level sets of distance-like functions with well controlled Laplacian. The argument can be thought as a highly non-linear version of the moving planes method, with no symmetries on the background and in a very low regularity setting.\\ 
The \emph{stability} of Laplacian bounds under the Hopf--Lax duality (equivalently, along solutions of the Hamilton--Jacobi equation), proved in \cite{MoS21} building on \cite{Kuwada07,AmbrosioGigliSavare15Bochner}, plays the role of the classical computation with Jacobi fields in Riemannian geometry. It crucially enters into play in the construction of the perturbations. The adimensional versions of \eqref{eq:sharplaplacianintro} can be improved to sharp dimensional bounds thanks to the well established localization technique \cite{CM17,CM20}. The additional difficulty with respect to the case of local perimeter minimizers considered in \cite{MoS21} is that only volume fixing perturbations are admissible for the isoperimetric problem.
\medskip

We do not claim that \autoref{thm:Isoperimetriciintro} carries all the information provided by the first and second variation formulas for the area in the smooth setting, however one of the main novelties of the present work will be to show that it is a valid replacement in several circumstances. Moreover, its global nature makes it more suitable for stability arguments.
Among its direct consequences we mention sharp Heintze--Karcher type bounds for perimeters \autoref{prop:variationofarea} and volumes \autoref{cor:volumebounds} of variations of isoperimetric sets via equidistant sets, that can be obtained by integration. 
\medskip

The ability to deduce information about the isoperimetric behaviour of an $\RCD(K,N)$ metric measure space $(X,\dist,\haus^N)$ from \autoref{thm:Isoperimetriciintro} is related to the existence of isoperimetric regions, which is not guaranteed, when $\haus^N(X)=\infty$, see, e.g., \cite[Example 3.6]{AFP21}, or the introduction
of \cite{Nar14}. \\ 
We overcome this issue thanks to the asymptotic mass decomposition result recently proved in \cite{AntonelliNardulliPozzetta}, extending the previous \cite{RitRosales04,Nar14,MondinoNardulli16,AFP21} and building on a concentration-compactness argument. If $(X,\dist)$ is compact, a minimizing sequence for the isoperimetric problem \eqref{eq:isoIntro} for volume $V>0$ converges, up to subsequences, to an isoperimetric set, by lower semicontinuity of the perimeter. This is not true, in general, when $(X,\dist)$ is not compact as elementary examples illustrate. However, \cite{AntonelliNardulliPozzetta} shows that, if $(X,\dist,\haus^N)$ is a non-compact $\RCD(K,N)$ space with $\haus^N(B_1(p))>v_0$ for any $p\in X$ for some $v_0>0$, then, up to subsequences, every minimizing sequence for the isoperimetric problem at a given volume $V>0$ splits into finitely many pieces. One of them converges to an isoperimetric region in $X$, possibly of volume less than $V$ (possibly zero). The others converge to isoperimetric regions inside pointed measured Gromov--Hausdorff limits at infinity $(Y,\dist_Y,\haus^N,q)$ of sequences $(X,\dist,\haus^N,p_i)$, where $\dist(p,p_i)\to\infty$ as $i\to\infty$ for some reference point $p\in X$. Moreover, there is no loss of total mass in this process, see \autoref{thm:MassDecompositionINTRO}, hence the result can be seen as a generalized existence of isoperimetric regions.
\medskip

An observation going back to \cite{Nar14} {(see also the subsequent \cite{MondinoNardulli16})} is that if $(X,\dist)$ and all its pointed limits at infinity are isometric to smooth Riemannian manifolds, then generalized isoperimetric regions can be used as isoperimetric regions in the compact case in combination with \eqref{eq:secondvariation} to deduce useful information. However, the assumption that $(X,\dist)$ is isometric to a smooth Riemannian manifold with $\mathrm{Ric}\ge K$ and $\haus^N(B_1(p))>v_0$, for every $p\in X$ and some $v_0>0$, does not guarantee any regularity of its pointed limits at infinity besides them being non-collapsed Ricci limit spaces.

The ability to estimate the second variation of the perimeter via equidistant sets for isoperimetric sets in general $\RCD(K,N)$ spaces $(X,\dist,\haus^N)$ as in \autoref{thm:Isoperimetriciintro} makes this heuristic work without further assumptions in the smooth Riemannian case and in greater generality, obtaining \autoref{thm:BavardPansuIntro}. This is a fundamental new contribution of the present work.
\medskip

{Let us briefly comment on the possible extensions of the present work to more general settings.\\
The generalization of (some of) the results obtained in this paper to general $\RCD(K,N)$ metric measure spaces $(X,\dist,\meas)$ with $N<\infty$ is left to the future investigation, due to some additional difficulties with respect to the case $\meas=\haus^N$ considered here that we outline below.}

The two main ingredients for the proof of \autoref{thm:BavardPansuIntro} are the asymptotic mass decomposition \autoref{thm:MassDecompositionINTRO} and the Laplacian comparison \autoref{thm:Isoperimetriciintro}.
The asymptotic mass decomposition \autoref{thm:MassDecompositionINTRO} has been obtained in \cite{AntonelliNardulliPozzetta} by the first and the third authors together with Nardulli, leveraging on the techniques developed in \cite{Nar14, AFP21}. We expect that the results of \cite{AntonelliNardulliPozzetta} generalize to arbitrary $\RCD(K,N)$ metric measure spaces $(X,\dist,\meas)$. However, when $\meas=\haus^N$, in the proof of \cite{AntonelliNardulliPozzetta} we exploit the fact that the density of the measure is $1$ almost everywhere, and the perimeters of the balls centered at almost every point are infinitesimally equivalent to the perimeters of the balls in $\R^N$. Both these statement fail in general when $\meas\neq\haus^N$.

We also expect that \autoref{thm:Isoperimetriciintro} holds for arbitrary reference measures $\meas$. In particular, we notice that the Laplacian comparison holds for isoperimetric sets in smooth weighted Riemannian manifolds verifying the $\CD(K,N)$ condition. However, the extension would require again some new insights, in particular in reference to the mild topological regularity for isoperimetric sets obtained in \cite{AntonelliPasqualettoPozzetta21}. 
\medskip

{
Besides the case of general $\RCD(K,N)$ metric measure spaces $(X,\dist,\meas)$ with finite $N$, two natural directions for the future investigation are the infinite-dimensional and the non-linear settings.

In the case of $\RCD(K,\infty)$ metric measure spaces we raise the following:

\begin{conjecture}\label{conjinfty}
    Let $(X,\dist,\meas)$ be an $\RCD(K,\infty)$ metric measure space. Then the isoperimetric profile $I$ satisfies the second order differential inequality
    \begin{equation}\label{conj:infty}
        -I^{''}I\ge K 
    \end{equation}
    in the viscosity sense on its domain.
\end{conjecture}

We remark that the inequality \eqref{conj:infty} would be saturated by the Euclidean space endowed with the standard metric and a (suitably normalized) weighted Gaussian measure.

}


\medskip

{

In the direction of removing the Hilbertian assumption, to the best of our knowledge, the validity of sharp differential second-order inequalities for the isoperimetric profile has not been investigated before even in the case of smooth Finsler manifolds satisfying curvature-dimension bounds. In this regard we ask the following:

\begin{question}\label{qCD}
Do the Laplacian comparison for the distance from isoperimetric sets \autoref{thm:Isoperimetriciintro} and the
sharp second-order differential inequalities for the isoperimetric profile \autoref{thm:BavardPansuIntro} hold for
(possibly essentially non-branching) $\CD(K,N)$ metric measure spaces? 
\end{question}

Apart from the additional technical challenges with respect to the setting considered in the present paper, we believe that addressing \autoref{conjinfty} and \autoref{qCD} might require the development of new strategies.

}


\subsection*{Comparison with the previous literature}

We conclude this introduction with a brief comparison between our results and the previous literature about the isoperimetric problem under lower curvature bounds, without the aim of being comprehensive.
\begin{itemize}

    \item The difficulty of obtaining second order properties for the isoperimetric profile on \emph{non-smooth spaces} was pointed out in \cite[page 99]{Bayle03}, \cite{LedouxCtoIso} and in \cite[Appendix]{Milmanconvexity}. To the best of our knowledge, \autoref{thm:BavardPansuIntro} is the first instance of second order properties of the isoperimetric profile in a context where no approximation with smooth Riemannian manifolds is at disposal.
       \item The setting of $\RCD(0,N)$ spaces $(X,\dist,\haus^N)$ recovers in particular many of the results for Euclidean convex cones treated in \cite{RitRosales04} and for cones with non-negative Ricci curvature considered in \cite{MorganRitore02}.
    \item The results of the present paper recover, in a more general setting, many of the results proved in \cite{LeonardiRitore} for unbounded Euclidean convex bodies of uniform geometry.\\
    The setting of $\RCD(K,N)$ spaces $(X,\dist,\haus^N)$ includes, more in general, convex subsets of smooth Riemannian manifolds with Ricci curvature bounded from below, regardless of any compactness assumption and regularity assumption on the boundary. Compact convex subsets of Riemannian manifolds with Ricci curvature lower bounds have been considered in \cite{BayleRosales}. 
    \item The stability of mean curvature barriers in the sense of \autoref{thm:Isoperimetriciintro} for Gromov--Hausdorff converging sequences of boundaries inside measured Gromov--Hausdorff converging sequences of $\RCD(K,N)$ spaces has been recently observed in \cite{Ketterer21} (see also the previous \cite{BrueNaberSemola20}). In this regards, The main novelty of our work is to provide a large and natural class of sets having mean curvature barriers, namely isoperimetric sets. Moreover, we prove that $L^1$ convergence (which is guaranteed, up to subsequences, for equibounded isoperimetric sets) self-improves to Gromov--Hausdorff convergence under very natural assumptions. Therefore the stability of mean curvature barriers applies in this setting.
\end{itemize}

In order to put things into perspective, we stress that this paper heavily relies on the results of \cite{AntonelliPasqualettoPozzetta21, AntonelliNardulliPozzetta} (mild regularity of isoperimetric sets and asymptotic mass decomposition, respectively), while it is independent of the existence results in \cite{AntBruFogPoz} by the first and third authors together with Bruè and Fogagnolo, and it is completely independent of the sharp isoperimetric inequality obtained in \cite{BrendleFigo, BaloghKristaly}. The key contribution of the present work is to develop some refined tools of geometric measure theory in a low regularity setting and to combine them in original way with the asymptotic mass decomposition from \cite{AntonelliNardulliPozzetta}, obtaining several consequences for the isoperimetric problem under lower Ricci curvature bounds.

\subsection*{Addendum}
This is the first of two companion papers, together with \cite{AntonelliPasqualettoPozzettaSemolaASYMPTOTIC}. The joint version of the two papers appeared on arXiv in \cite{ConcavitySharpAPPS}. In the second one we are going to explore several consequences of the main results of this paper to the study of asymptotic isoperimetry on non-collapsed spaces with Ricci lower bounds, especially in the case where $K=0$:
\begin{itemize}
    \item exploiting \autoref{thm:Isoperimetriciintro}, we give a new proof of the sharp isoperimetric inequality on $\RCD(0,N)$ spaces $(X,\dist,\haus^N)$ with Euclidean volume growth that has been considered, with increasing level of generality, in \cite{AgostinianiFogagnoloMazzieri, FogagnoloMazzieri, BrendleFigo, BaloghKristaly, CavallettiManini}. We prove also the rigidity of the inequality in the setting above without any additional regularity assumption. Namely we prove that equality is achieved if and only if the ambient space is isometric to a cone and the set saturating the inequality is isometric to a ball centered at a tip of the cone;
    
    \item we explicitly determine the asymptotic behavior of the isoperimetric profile for volumes tending to zero on $\RCD(K,N)$ spaces $(X,\dist,\haus^N)$, see \autoref{rem:AsintoticaProfileSmallVolumes};
    
    \item we analyze the behavior of sequences of isoperimetric sets for volumes tending to zero on $\RCD(K,N)$ spaces $(X,\dist,\haus^N)$ and for volumes tending to infinity on Alexandrov spaces with non-negative curvature and Euclidean volume growth.
\end{itemize}

\subsection*{Acknowledgements}

The first author is partially supported by the European Research Council (ERC Starting Grant 713998 GeoMeG `\emph{Geometry of Metric Groups}').
The second author is supported by the Balzan project led by Luigi Ambrosio. The first and the third authors are partially supported by the INdAM - GNAMPA Project CUP\_E55F22000270001 ``Isoperimetric problems: variational and geometric aspects''.
The last author is supported by the European Research Council (ERC), under the European Union Horizon 2020 research and innovation programme, via the ERC Starting Grant  “CURVATURE”, grant agreement No. 802689. He is grateful to Elia Bru\`e and Andrea Mondino for several conversations and useful comments on a preliminary version of this note and to Christian Ketterer for useful comments.
The authors are grateful to Otis Chodosh, Mattia Fogagnolo and Emanuel Milman for useful feedback on a preliminary version of the paper. {They are grateful to the anonymous reviewers for several useful comments that led to a great improvement in the presentation of the note and for pointing out some inaccuracies in a previous version.}

\medskip

\section{Preliminaries}

In this paper, by a \emph{metric measure space} (briefly, m.m.s.) we mean a triple \((X,\dist,\meas)\), where \((X,\dist)\) is a complete and separable
metric space, while \(\meas\geq 0\) is a boundedly-finite Borel measure on \(X\).
For any \(k\in[0,\infty)\), we denote by \(\haus^k\) the \emph{\(k\)-dimensional Hausdorff measure} on \((X,\dist)\).
We indicate with \(C(X)\) the space of all continuous functions \(f\colon X\to\R\) and \(C_b(X)\coloneqq\{f\in C(X)\,:\,f\text{ is bounded}\}\).
We denote by \({\rm LIP}(X)\subseteq C(X)\) the space of all Lipschitz functions, while \({\rm LIP}_{\mathrm{bs}}(X)\) (resp.\ \({\rm LIP}_c(X)\)) stands for
the set of all those \(f\in{\rm LIP}(X)\) whose support \({\rm spt}f\) is bounded (resp.\ compact). More generally, we denote by
\({\rm LIP}_{\mathrm{loc}}(X)\) the space of locally Lipschitz functions \(f\colon X\to\setR\). Given $f\in{\rm LIP}_{\mathrm{loc}}(X)$,
\[
\lip f (x) \eqdef \limsup_{y\to x} \frac{|f(y)-f(x)|}{\dist(x,y)}
\]
is the \emph{slope} of $f$ at $x$, for any accumulation point $x\in X$, and $\lip f(x)\coloneqq 0$ if $x\in X$ is isolated.

We shall also work with the local versions of the above spaces: given \(\Omega\subseteq X\) open, we will consider the spaces
\({\rm LIP}_c(\Omega)\subseteq{\rm LIP}_{\mathrm{bs}}(\Omega)\subseteq{\rm LIP}(\Omega)\subseteq{\rm LIP}_{\mathrm{loc}}(\Omega)\),
where by \({\rm LIP}_{\mathrm{bs}}(\Omega)\) we mean the space of all \(f\in{\rm LIP}(\Omega)\) having bounded support \({\rm spt}f\subseteq\Omega\)
that verifies \(\dist({\rm spt}f,\partial\Omega)>0\).
\medskip

Let us define
\[
\sn_K(r) := \begin{cases}
(-K)^{-\frac12} \sinh((-K)^{\frac12} r) & K<0\, ,\\
r & K=0\, ,\\
K^{-\frac12} \sin(K^{\frac12} r) & K>0\, .
\end{cases}
\]
We denote by $v(N,K,r)$ and $s(N, K, r)$ the volume and the surface measure, respectively, of the ball of radius $r$ in the (unique) simply connected Riemannian manifold of sectional curvature $K$ and dimension $N$.
In particular $s(N,K,r)=N\omega_N\mathrm{sn}_K^{N-1}(r)$ and $v(N,K,r)=\int_0^rN\omega_N\mathrm{sn}_K^{N-1}(r)\de r$, where $\omega_N$ is the Euclidean volume of the Euclidean unit ball in $\mathbb R^N$.
\medskip

\subsection{Convergence and stability results}
{
Following the exposition of \cite{AmbrosioBrueSemola19}, we introduce a definition of pointed measured
Gromov--Hausdorff convergence (via a proper realization) that is fit for our purposes. In our setting, where
we always deal with locally uniformly doubling measures, this definition is equivalent to the standard one,
see \cite[Theorem 3.15 and Section 3.5]{GigliMondinoSavare15}.
}

%
%
\begin{definition}[pGH and pmGH convergence]\label{def:GHconvergence}
A sequence $\{ (X_i, \dist_i, x_i) \}_{i\in \N}$ of pointed metric spaces is said to converge in the \emph{pointed Gromov--Hausdorff topology, in the $\mathrm{pGH}$ sense for short,} to a pointed metric space $ (Y, \dist_Y, y)$ if there exist a complete separable metric space $(Z, \dist_Z)$ and isometric embeddings
\[
\begin{split}
&\Psi_i:(X_i, \dist_i) \to (Z,\dist_Z), \qquad \forall\, i\in \N\, ,\\
&\Psi:(Y, \dist_Y) \to (Z,\dist_Z)\, ,
\end{split}
\]
such that for any $\eps,R>0$ there is $i_0(\varepsilon,R)\in\mathbb N$ such that
\[
\Psi_i(B_R^{X_i}(x_i)) \subset \left[ \Psi(B_R^Y(y))\right]_\eps,
\qquad
\Psi(B_R^{Y}(y)) \subset \left[ \Psi_i(B_R^{X_i}(x_i))\right]_\eps\, ,
\]
for any $i\ge i_0$, where $[A]_\eps\coloneqq \{ z\in Z \st \dist_Z(z,A)\leq \eps\}$ for any $A \subset Z$.

Let $\meas_i$ and $\mu$ be given in such a way $(X_i,\dist_i,\meas_i,x_i)$ and $(Y,\dist_Y,\mu,y)$ are m.m.s.\! If in addition to the previous requirements we also have $(\Psi_i)_\sharp\mathfrak{m}_i \rightharpoonup \Psi_\sharp \mu$ with respect to duality with continuous bounded functions on $Z$ with bounded support, then the convergence is said to hold in the \emph{pointed measured Gromov--Hausdorff topology, or in the $\mathrm{pmGH}$ sense for short}.
\end{definition}

We need to recall a generalized $L^1$-notion of convergence for sets defined on a sequence of metric measure spaces converging in the pmGH sense. Such a definition is given in \cite[Definition 3.1]{AmbrosioBrueSemola19}, and it is investigated in \cite{AmbrosioBrueSemola19} capitalizing on the results in \cite{AmbrosioHonda17}.

\begin{definition}[$L^1$-strong and $L^1_{\mathrm{loc}}$ convergence]\label{def:L1strong}
Let $\{ (X_i, \dist_i, \mathfrak{m}_i, x_i) \}_{i\in \N}$  be a sequence of pointed metric measure spaces converging in the pmGH sense to a pointed metric measure space $ (Y, \dist_Y, \mu, y)$ and let $(Z,\dist_Z)$ be a realization as in \autoref{def:GHconvergence}.

We say that a sequence of Borel sets $E_i\subset X_i$ such that $\mathfrak{m}_i(E_i) < +\infty$ for any $i \in \N$ converges \emph{in the $L^1$-strong sense} to a Borel set $F\subset Y$ with $\mu(F) < +\infty$ if $\mathfrak{m}_i(E_i) \to \mu(F)$ and $\chi_{E_i}\mathfrak{m}_i \rightharpoonup \chi_F\mu$ with respect to the duality with continuous bounded functions with bounded support on $Z$.

We say that a sequence of Borel sets $E_i\subset X_i$ converges \emph{in the $L^1_{\mathrm{loc}}$-sense} to a Borel set $F\subset Y$ if $E_i\cap B_R(x_i)$ converges to $F\cap B_R(y)$ in $L^1$-strong for every $R>0$.
\end{definition}

\begin{definition}[Hausdorff convergence]
Let $\{ (X_i, \dist_i, \mathfrak{m}_i, x_i) \}_{i\in \N}$  be a sequence of pointed metric measure spaces converging in the pmGH sense to a pointed metric measure space $ (Y, \dist_Y, \mu, y)$.
Then we say that a sequence of closed sets $E_i\subset X_i$ converges \emph{in Hausdorff distance} (or \emph{in Hausdorff sense}) to a closed set $F\subset Y$ if there holds convergence in Hausdorff distance in a realization $(Z,\dist_Z)$ of the pmGH convergence as in \autoref{def:GHconvergence}.
\end{definition}

It is also possible to define notions of uniform convergence and $H^{1,2}$-strong and weak convergences for sequences of functions of a sequence of spaces $X_i$ converging in pointed measure Gromov--Hausdorff sense. We refer the reader to \cite{AmbrosioBrueSemola19, AmbrosioHonda17} for such definitions.

\subsection{\texorpdfstring{$\rm BV$}{BV} functions and sets of finite perimeter in metric measure spaces}
We begin with the definitions of \emph{function of bounded variation} and \emph{set of finite perimeter} in a m.m.s.
\begin{definition}[$\rm BV$ functions and perimeter on m.m.s.]\label{def:BVperimetro}
Let $(X,\dist,\meas)$ be a metric measure space.  Given $f\in L^1_{\mathrm{loc}}(X,\meas)$ we define
\[
|Df|(A) \eqdef \inf\left\{\liminf_i \int_A \lip f_i \de\meas \st \text{$f_i \in {\rm LIP}_{\mathrm{loc}}(A),\,f_i \to f $ in $L^1_{\mathrm{loc}}(A,\meas)$} \right\}\, ,
\]
for any open set $A\subseteq X$.
We declare that a function \(f\in L^1_{\mathrm{loc}}(X,\meas)\) is of \emph{local bounded variation}, briefly \(f\in{\rm BV}_{\mathrm{loc}}(X)\),
if \(|Df|(A)<+\infty\) for every \(A\subseteq X\) open bounded.
A function $f \in L^1(X,\meas)$ is said to be of \emph{bounded variation}, briefly $f\in{\rm BV}(X)$, if $|Df|(X)<+\infty$. 

If $E\subseteq\X$ is a Borel set and $A\subseteq X$ is open, we  define the \emph{perimeter $\Per(E,A)$  of $E$ in $A$} by
\[
\Per(E,A) \eqdef \inf\left\{\liminf_i \int_A \lip u_i \de\meas \st \text{$u_i \in {\rm LIP}_{\mathrm{loc}}(A),\,u_i \to \nchi_E $ in $L^1_{\mathrm{loc}}(A,\meas)$} \right\}\, ,
\]
in other words \(\Per(E,A)\coloneqq|D\nchi_E|(A)\).
We say that $E$ has \emph{locally finite perimeter} if $\Per(E,A)<+\infty$ for every open bounded set $A$. We say that $E$ has \emph{finite perimeter} if $\Per(E,X)<+\infty$, and we denote $\Per(E)\eqdef \Per(E,X)$.
\end{definition}

\medskip

In the sequel, we shall frequently make use of the following \emph{coarea formula}, proved in \cite{Miranda03}.
\begin{theorem}[Coarea formula]\label{thm:coarea}
Let \((X,\dist,\meas)\) be a metric measure space. Fix \(f\in L^1_{\mathrm{loc}}(X)\) and an open set \(\Omega\subseteq X\).
Then \(\R\ni t\mapsto \Per(\{f>t\},\Omega)\in[0,+\infty]\) is Borel measurable and
\[
|Df|(\Omega)=\int_\R \Per(\{f>t\},\Omega)\,\de t\, .
\]
\end{theorem}

\subsection{Sobolev functions, Laplacians and vector fields in metric measure spaces}
The \emph{Cheeger energy} on a metric measure space \((X,\dist,\meas)\) is defined as the \(L^2\)-relaxation of the functional
\(f\mapsto\frac{1}{2}\int\lip^2 f_n\di\meas\) (see \cite{AmbrosioGigliSavare11} after \cite{Cheeger99}). Namely, for any function \(f\in L^2(X)\) we define
\[
\Ch(f)\coloneqq\inf\bigg\{\liminf_{n\to\infty} \frac{1}{2}\int\lip^2 f_n\di\meas\;\bigg|\;(f_n)_{n\in\setN}\subseteq{\rm LIP}_{\mathrm{bs}}(X),\,f_n\to f\text{ in }L^2(X)\bigg\}\, .
\]
The \emph{Sobolev space} \(H^{1,2}(X)\) is defined as the finiteness domain \(\{f\in L^2(X)\,:\,\Ch(f)<+\infty\}\) of the Cheeger energy.
The restriction of the Cheeger energy to the Sobolev space admits the integral representation \(\Ch(f)=\frac{1}{2}\int|\nabla f|^2\di\meas\),
for a uniquely determined function \(|\nabla f|\in L^2(X)\) that is called the \emph{minimal weak upper gradient} of \(f\in H^{1,2}(X)\).
The linear space \(H^{1,2}(X)\) is a Banach space if endowed with the Sobolev norm
\[
\|f\|_{H^{1,2}(X)}\coloneqq\sqrt{\|f\|_{L^2(X)}^2+2\Ch(f)}=\sqrt{\|f\|_{L^2(X)}^2+\||\nabla f|\|_{L^2(X)}^2},\quad\text{ for every }f\in H^{1,2}(X)\, .
\]
Following \cite{Gigli12}, when \(H^{1,2}(X)\) is a Hilbert space (or equivalently \(\Ch\) is a quadratic form) we say that the metric measure
space \((X,\dist,\meas)\) is \emph{infinitesimally Hilbertian}.
\medskip

Hereafter, the infinitesimal Hilbertianity of \((X,\dist,\meas)\) will be our standing assumption.

The results of \cite{AmbrosioGigliSavare11_3} ensure that \({\rm LIP}_{\mathrm{bs}}(X)\) is dense in \(H^{1,2}(X)\) with respect to the
norm topology. We define the bilinear mapping \(H^{1,2}(X)\times H^{1,2}(X)\ni(f,g)\mapsto\nabla f\cdot\nabla g\in L^1(X)\) as
\[
\nabla f\cdot\nabla g\coloneqq\frac{|\nabla(f+g)|^2-|\nabla f|^2-|\nabla g|^2}{2},\quad\text{ for every }f,g\in H^{1,2}(X)\, .
\]
Given \(\Omega\subseteq X\) open, we define the \emph{local Sobolev space with Dirichlet boundary conditions} \(H^{1,2}_0(\Omega)\) as the closure
of \({\rm LIP}_{\mathrm{bs}}(\Omega)\) in \(H^{1,2}(X)\). Notice that \(H^{1,2}_0(X)=H^{1,2}(X)\). Moreover, we declare that a given function \(f\in L^2(\Omega)\)
belongs to the \emph{local Sobolev space} \(H^{1,2}(\Omega)\) provided \(\eta f\in H^{1,2}(X)\) for every \(\eta\in{\rm LIP}_{\mathrm{bs}}(\Omega)\) and {the function}
\[
|\nabla f|\coloneqq{\rm ess\,sup}\big\{\chi_{\{\eta=1\}}|\nabla(\eta f)|\;\big|\;\eta\in{\rm LIP}_{\mathrm{bs}}(\Omega)\big\}
\]
{belongs to $L^2(X)$.}
{Above, we employed the notation} \({\rm ess\,sup}_{\lambda\in\Lambda}h_\lambda\) to denote the essential supremum of a set
\(\{h_\lambda\}_{\lambda\in\Lambda}\) of measurable functions, {and denoted by $\chi$ the indicator function.}

\begin{definition}[Local Laplacian]
Let \((X,\dist,\meas)\) be an infinitesimally Hilbertian space and \(\Omega\subseteq X\) an open set. Then we say that a function
\(f\in H^{1,2}(\Omega)\) has \emph{local Laplacian} in \(\Omega\), \(f\in D(\Delta,\Omega)\) for short, if there exists a (uniquely determined)
function \(\Delta f\in L^2(\Omega)\) such that
\[
\int_\Omega g\Delta f\di\meas=-\int_\Omega\nabla g\cdot\nabla f\di\meas,\quad\text{ for every }g\in H^{1,2}_0(\Omega)\, .
\]
For brevity, we write \(D(\Delta)\) instead of \(D(\Delta,X)\).
\end{definition}
More generally, we work with functions having measure-valued Laplacian in an open set:
%
%
\begin{definition}[Measure-valued Laplacian]
Let \((X,\dist,\meas)\) be an infinitesimally Hilbertian space and \(\Omega\subseteq X\) an open set. Then we say that a function
\(f\in H^{1,2}(\Omega)\) has \emph{measure-valued Laplacian} in \(\Omega\), \(f\in D(\mathbf\Delta,\Omega)\) for short, provided there
exists a (uniquely determined) locally finite measure \(\mathbf\Delta f\) on \(\Omega\) such that
\[
\int_\Omega g\mathbf\Delta f\coloneqq\int_\Omega g\di\mathbf\Delta f=-\int_\Omega\nabla g\cdot\nabla f\di\meas\, ,\quad\text{ for every }g\in{\rm LIP}_{\mathrm{bs}}(\Omega)\, .
\]
For brevity, we write \(D(\mathbf\Delta)\) instead of \(D(\mathbf\Delta,X)\). Moreover, given functions \(f\in{\rm LIP}(\Omega)\cap H^{1,2}(\Omega)\)
and \(\eta\in C_b(\Omega)\), we say that \emph{\(\mathbf\Delta f\leq\eta\) in the distributional sense} if \(f\in D(\mathbf\Delta,\Omega)\) and
\(\mathbf\Delta f\leq\eta\meas\).
\end{definition}
The above two notions of Laplacian are consistent, in the following sense: given any \(f\in H^{1,2}(\Omega)\), it holds that \(f\in D(\Delta,\Omega)\)
if and only if \(f\in D(\mathbf\Delta,\Omega)\), \(\mathbf\Delta f\ll\meas\) and \(\frac{\di\mathbf\Delta f}{\di\meas}\in L^2(\Omega)\). If this is
the case, then we also have that the \(\meas\)-a.e.\ equality \(\Delta f=\frac{\di\mathbf\Delta f}{\di\meas}\) holds.
\medskip

The \emph{heat flow} \(\{P_t\}_{t\geq 0}\) on \((\X,\dist,\meas)\) is the gradient flow of the quadratic form \(\Ch\) in \(L^2(X)\). For any \(f\in L^2(X)\),
the gradient flow trajectory \([0,+\infty]\ni t\mapsto P_t f\in L^2(X)\) is a continuous curve with \(P_0 f=f\) that is locally absolutely continuous
in \((0,+\infty)\), and with \(P_t f\in D(\Delta)\) and \(\frac{\di}{\di t}P_t f=\Delta P_t f\) for a.e.\ \(t>0\). 
\medskip

By a \emph{bounded Sobolev derivation} on \((X,\dist,\meas)\) we mean a linear operator \(v\colon H^{1,2}(X)\to L^2(X)\) for which there exists
a function \(g\in L^\infty(X)\) satisfying \(|v(f)|\leq g|\nabla f|\) \(\meas\)-a.e.\ for every \(f\in H^{1,2}(X)\). The minimal (in the \(\meas\)-a.e.\ sense)
function \(g\) verifying this condition is denoted by \(|v|\in L^\infty(X)\) and called the \emph{pointwise norm} of \(v\).
We then define the space \(L^\infty(TX)\) of \emph{bounded vector fields} on \((X,\dist,\meas)\) as the family of all bounded Sobolev derivations
on \((X,\dist,\meas)\). The space \(L^\infty(TX)\) is a module over \(L^\infty(X)\) if endowed with the multiplication operator
\(L^\infty(X)\times L^\infty(TX)\ni(h,v)\mapsto h\cdot v\in L^\infty(TX)\) given by \((h\cdot v)(f)\coloneqq h v(f)\) for every \(f\in H^{1,2}(X)\).
Moreover, to any given function \(f\in H^{1,2}(X)\) with \(|\nabla f|\in L^\infty(X)\) we can associate its \emph{gradient} \(\nabla f\in L^\infty(TX)\),
which is characterized as the unique element of \(L^\infty(TX)\) satisfying \(\nabla f(g)=\nabla f\cdot\nabla g\) for every \(g\in H^{1,2}(X)\).
In particular, to use the notation
\[
v\cdot\nabla f\coloneqq v(f)\, ,\quad\text{ for every }v\in L^\infty(TX)\text{ and }f\in H^{1,2}(X)
\]
will cause no ambiguity. Observe also that the pointwise norm of \(\nabla f\) de facto coincides with the minimal weak upper gradient \(|\nabla f|\) of \(f\).
The above notions are essentially borrowed from \cite{Gigli14,Gigli17}, up to some technical subtleties one can easily figure out and deal with.
%
%
%
\begin{definition}[Essentially bounded divergence measure vector fields]
Let \((X,\dist,\meas)\) be an infinitesimally Hilbertian space. Then we say that an element \(v\in L^\infty(TX)\)
is an \emph{essentially bounded divergence measure vector field} if there exists a (uniquely determined) finite Radon measure
\({\rm div}(v)\) on \(X\) such that
\[
\int g\,{\rm div}(v)\coloneqq\int g\di({\rm div}(v))=-\int v\cdot\nabla g\di\meas\quad\text{ for every }g\in{\rm LIP}_c(X)\, .
\]
We denote by \(\mathcal{DM}^\infty(X)\) the family of all essentially bounded divergence measure vector fields.
\end{definition}
Similarly, one can also define the space \(\mathcal{DM}^\infty(\Omega)\) of locally essentially bounded divergence measure vector fields
in some open set \(\Omega\subseteq X\). Notice that, given any function \(f\in H^{1,2}(X)\) with \(|\nabla f|\in L^\infty(X)\),
it holds that \(\nabla f\in L^\infty(TX)\) is an essentially bounded divergence measure vector field if and only if \(f\in D(\mathbf\Delta)\)
and \(\mathbf\Delta f\) is finite. If this is the case, then we also have that \({\rm div}(\nabla f)=\mathbf\Delta f\). The analogous
property holds for Sobolev functions defined on an open set \(\Omega\).
\subsection{Geometric analysis on \texorpdfstring{$\RCD$}{RCD} spaces}\label{sec:RCD}

The focus of this paper will be on $\RCD(K,N)$ metric measure spaces $(X,\dist,\meas)$. 
We avoid giving a detailed introduction to this notion, addressing the reader to the survey \cite{AmbrosioSurvey} and references therein for the relevant background. 

For most of the results of this paper we will consider \(\RCD(K,N)\) spaces of the form \((X,\dist,\haus^N)\), for some \(K\in\setR\)
and \(N\in\N\). Notice that we are requiring that the dimension of the Hausdorff measure coincides with the upper dimensional bound
in the \(\RCD\) condition. These spaces are typically called \emph{non-collapsed \(\RCD\) spaces} (\({\rm ncRCD}(K,N)\) spaces for short) or $N$--dimensional $\RCD(K,N)$ spaces
(see \cite{Kitabeppu17,DePhilippisGigli18,MondinoKapovitch}).
\medskip

Below we recall some of the {less classical} properties that will be relevant for our purposes.
\medskip

In the setting of $\RCD(K,N)$ spaces it is possible to employ a viscosity interpretation of Laplacian bounds, in addition to the distributional one, see \cite{MoS21}.

\begin{definition}[Bounds in the viscosity sense for the Laplacian]\label{def:viscosity}
Let $(X,\dist,\meas)$ be an $\RCD(K,N)$ metric measure space and let $\Omega\subset X$ be an open and bounded domain. Let $f:\Omega\to\setR$ be locally Lipschitz and $\eta\in\Cb(\Omega)$. We say that $\Delta f\le \eta$ in the \textit{viscosity sense} in $\Omega$ if the following holds. For any $\Omega'\Subset\Omega$ and for any test function $\phi:\Omega'\to\setR$ such that 
\begin{itemize}
\item[(i)] $\phi\in D(\Delta, \Omega')$ and $\Delta\phi$ is continuous on $\Omega'$;
\item[(ii)] for some $x\in \Omega'$ it holds 
$\phi(x)=f(x)$ and $\phi(y)\le f(y)$ for any $y\in\Omega'$, $y\neq x$;
\end{itemize}
it holds
\begin{equation*}
\Delta \phi(x)\le \eta(x)\, .
\end{equation*} 
{ A function $\varphi$ as in items (i) and (ii) above will be called {\em lower supporting function of $f$}. When instead of $\leq$ we consider $\geq$ in the definition above, a function $\varphi$ as in items (i) and (ii) will be called {\em upper supporting function of $f$}.}
\end{definition}

We will rely on the equivalence between distributional and viscosity bounds on the Laplacian. The result is classical in the setting of smooth Riemannian manifolds and it has been extended in \cite[Theorem 3.24]{MoS21} to $\RCD(K,N)$ metric measure spaces $(X,\dist,\haus^N)$.

\begin{theorem}\label{thm:viscoimpldistri}
Let $(X,\dist,\haus^N)$ be an $\RCD(K,N)$ metric measure space. Let $\Omega\subset X$ be an open and bounded domain, $f:\Omega\to\setR$ be a Lipschitz function and $\eta:\Omega\to\setR$ be continuous. Then $\boldsymbol{\Delta} f\le \eta$ in the sense of distributions if and only if $\Delta f\le \eta$ in the viscosity sense.
\end{theorem}

We refer also to \cite[Theorem 3.28]{MoS21} for other equivalent characterizations of bounds on the Laplacian in the setting of $\RCD$ spaces.

In the proof of \autoref{thm:Isoperimetriciintro} it will be important to relate the synthetic lower Ricci curvature bound to the stability of Laplacian bounds through the Hopf-Lax duality
\begin{equation}
    f^c(x):=\inf_{y\in X}\{f(y)+\dist(x,y)\}\, .
\end{equation}

The following statement corresponds to \cite[Theorem 4.9]{MoS21}. On a smooth Riemannian manifold, neglecting the regularity issues, it follows from the two-points Laplacian comparison proved in \cite{AndrewsClut14}, which is based on a computation with Jacobi fields. The proof in \cite{MoS21}, which works in a much more general setting, is based on the interplay between the Hopf-Lax semigroup and the heat flow put forward in \cite{Kuwada07,GigliKuwadaOhta,AmbrosioGigliSavare15Bochner}.

\begin{theorem}\label{thm:HLlapla}
Let $(X,\dist,\haus^N)$ be an $\RCD(K,N)$ metric measure space for some $K\in\setR$ and $1\le N<\infty$. Let $f:X\to\setR$ be a locally Lipschitz function with polynomial growth. Let $\Omega,\Omega'\subset X$ be open domains and $\eta\in\setR$. Then the following holds. Assume that $f^c$ is finite and that, for any $x\in\Omega'$ the infimum defining $f^c(x)$ is attained at some $y\in \Omega$. Assume also that
\begin{equation}\label{eq:lapassumption}
\Delta f\le \eta \quad\text{on $\Omega$}\,.
\end{equation}
Then
\begin{align*}
&\Delta f^c\le \eta -K\max_{x\in\Omega', y\in\Omega}\dist(x,y) \quad\text{on $\Omega'$},\quad  \text{if $K\le 0$, } \\
&\Delta f^c\le \eta -K\min_{x\in\Omega', y\in\Omega}\dist(x,y) \quad\text{on $\Omega'$},  \quad  \text{if $K\ge 0$,}
\end{align*}
 where the Laplacian bounds are intended either in the distributional or in the viscosity sense.
\end{theorem}

\medskip
Given a Borel set \(E\subseteq X\) in an \(\RCD(K,N)\) space \((X,\dist,\haus^N)\) and any \(t\in[0,1]\), we denote by \(E^{(t)}\) the set
of \emph{points of density \(t\)} of \(E\), namely
\[
E^{(t)}\coloneqq\bigg\{x\in X\;\bigg|\;\lim_{r\to 0}\frac{\haus^N(E\cap B_r(x))}{\haus^N(B_r(x))}=t\bigg\}\, .
\]
The \emph{essential boundary} of \(E\) is defined as \(\partial^e E\coloneqq X\setminus(E^{(0)}\cup E^{(1)})\). We have that \(E^{(t)}\)
and \(\partial^e E\) are Borel sets. Furthermore, the \emph{reduced boundary} \(\mathcal F E\subseteq\partial^e E\) of a given set \(E\subseteq X\)
of finite perimeter is defined as the set of those points of \(X\) where the unique tangent to \(E\) (up to isomorphism) is the
half-space \(\{x=(x_1,\ldots,x_N)\in\setR^N\,:\,x_N>0\}\) in \(\setR^N\); see \cite[Definition 4.1]{AmbrosioBrueSemola19} for the precise
notion of convergence we are using. 

It was proved in \cite{BruePasqualettoSemola} after \cite{Ambrosio02,AmbrosioBrueSemola19} that the perimeter measure \(\Per(E,\cdot)\)
can be represented as
\begin{equation}\label{eq:RepresentationPerimeter}
\Per(E,\cdot)=\haus^{N-1}|_{\mathcal F E}.
\end{equation}

As it is evident from \eqref{eq:RepresentationPerimeter}, the notion of perimeter that we are using does not charge the boundary of the space under
consideration, if any.

We refer to \cite{DePhilippisGigli18,MondinoKapovitch,BrueNaberSemola20} for the relevant background about boundaries of $\RCD(K,N)$ spaces $(X,\dist,\haus^N)$.

Moreover, we recall that, according to \cite[Proposition 4.2]{BPSGaussGreen}, 
\[
\mathcal F E=E^{(1/2)}=\bigg\{x\in X\;\bigg|\;\lim_{r\to 0}\frac{\haus^N(E\cap B_r(x))}{\haus^N(B_r(x))}=\frac{1}{2}\bigg\}\, ,
\quad\text{ up to }\mathcal H^{N-1}\text{-null sets}\, .
\]

We will rely on a Gauss--Green integration by parts formula for essentially bounded divergence measure fields over sets of finite perimeter.


 As shown in \cite[Section 5]{BPSGaussGreen} after \cite{BuffaComiMiranda}, for any \(v\in\mathcal{DM}^\infty(X)\) and any set of finite perimeter $E$,
there exist measures \(D\chi_E(\chi_E v)\), \(D\chi_E(\chi_{E^c}v)\) on \(X\) such that \(((\chi_E v)\cdot\nabla P_t\chi_E)\meas
\rightharpoonup D\chi_E(\chi_E v)\) and \(((\chi_{E^c}v)\cdot\nabla P_t\chi_E)\meas\rightharpoonup D\chi_E(\chi_{E^c}v)\) as \(t\to 0\)
with respect to the narrow topology. Moreover, it holds that \(D\chi_E(\chi_E v),D\chi_E(\chi_{E^c}v)\ll\Per(E,\cdot)\). Then we define
\begin{equation}\label{eq:intextnorm}
(v\cdot\nu_E)_{\rm int}\coloneqq\frac{1}{2}\frac{\di(D\chi_E(\chi_E v))}{\di\Per(E,\cdot)}\, ,\quad
(v\cdot\nu_E)_{\rm ext}\coloneqq\frac{1}{2}\frac{\di(D\chi_E(\chi_{E^c}v))}{\di\Per(E,\cdot)}\, .
\end{equation}
We remark that there is full consistency with the classical setting {of Riemannian manifolds}.

%
%

The following result was proved in \cite[Theorem 5.2]{BPSGaussGreen} by building upon \cite[Theorem 6.22]{BuffaComiMiranda}.
\begin{theorem}[Gauss--Green]\label{thm:GaussGreen}
Let \((X,\dist,\haus^N)\) be an \(\RCD(K,N)\) space and \(E\subseteq X\) a set of finite perimeter.
Fix \(v\in\mathcal{DM}^\infty(X)\). Then \((v\cdot\nu_E)_{\rm int},(v\cdot\nu_E)_{\rm ext}\in L^\infty(\mathcal F E,\Per(E,\cdot))\) and
\[\begin{split}
\|(v\cdot\nu_E)_{\rm int}\|_{L^\infty(\mathcal F E,\Per(E,\cdot))}&\leq\|v\|_{L^\infty(E)},\\
\|(v\cdot\nu_E)_{\rm ext}\|_{L^\infty(\mathcal F E,\Per(E,\cdot))}&\leq\|v\|_{L^\infty(E^c)}.
\end{split}\]
Moreover, the \emph{Gauss--Green integration-by-parts formulas} hold: for any \(\varphi\in{\rm LIP}_c(X)\) it holds
\[\begin{split}
\int_{E^{(1)}}\varphi\,{\rm div}(v)+\int_E v\cdot\nabla\varphi\di\meas&=-\int_{\mathcal F E}\varphi(v\cdot\nu_E)_{\rm int}\di\Per(E,\cdot)\, ,\\
\int_{E^{(1)}\cap\mathcal F E}\varphi\,{\rm div}(v)+\int_E v\cdot\nabla\varphi\di\meas&=-\int_{\mathcal F E}\varphi(v\cdot\nu_E)_{\rm ext}\di\Per(E,\cdot)\, .
\end{split}\]
\end{theorem}

Next we report on the {natural} behaviour of interior/exterior normal traces over the boundary of the intersection of two finite perimeter sets,
{which has been investigated in \cite{BPSGaussGreen}.}


Given two sets of finite perimeter \(E,F\subseteq X\), it is well-known that also \(E\cap F\) has finite perimeter. The set
\(\{\nu_E=\nu_F\}\) introduced in \cite{BPSGaussGreen} can be characterized as
\[
\{\nu_E=\nu_F\}=\bigg\{x\in\mathcal F E\cap\mathcal F F\;\bigg|\;\lim_{r\to 0}\frac{\haus^N((E\cap F)\cap B_r(x))}{\haus^N(B_r(x))}=\frac{1}{2}\bigg\},
\quad\text{ up to }\haus^{N-1}\text{-null sets}\, ,
\]
as it follows from the results of \cite{BPSGaussGreen}. The ensuing statement is taken from \cite[Proposition 5.4]{BPSGaussGreen}.
\begin{proposition}\label{prop:5.4BPS}
Let \((X,\dist,\haus^N)\) be an \(\RCD(K,N)\) space. Let \(E,F\subseteq X\) be two sets of finite perimeter. Let \(v\in\mathcal{DM}^\infty(X)\) be given.
Then it holds that
\[\begin{split}
(v\cdot\nu_E)_{\rm int}=(v\cdot\nu_F)_{\rm int}&,\quad\haus^{N-1}\text{-a.e.\ on }\{\nu_E=\nu_F\}\, ,\\
(v\cdot\nu_{E\cap F})_{\rm int}=(v\cdot\nu_E)_{\rm int}&,\quad\Per(E,\cdot)\text{-a.e.\ in }F^{(1)}\, ,\\
(v\cdot\nu_{E\cap F})_{\rm int}=(v\cdot\nu_E)_{\rm int}&,\quad\haus^{N-1}\text{-a.e.\ on }\{\nu_E=\nu_F\}\, .
\end{split}\]
\end{proposition}
We will apply the previous machinery to the level sets of distance-type functions obtained through Hopf-Lax duality, see \autoref{thm:HLlapla}. More specifically, we will need the following result,
which is taken from \cite[Proposition 6.1]{BPSGaussGreen}. 

%
%
\begin{proposition}\label{prop:6.1BPS}
Let \((X,\dist,\haus^N)\) be an \(\RCD(K,N)\) space. Let \(\Omega\Subset\Omega'\subseteq X\) be open domains and let \(\varphi\colon\Omega'\to\setR\)
be a \(1\)-Lipschitz function. Suppose that \(|\nabla\varphi|=1\) holds \(\haus^N\)-a.e.\ on \(\Omega'\) and that there exists a constant \(L\leq 0\)
such that \(\mathbf\Delta\varphi\geq L\) in the sense of distributions in \(\Omega'\), thus in particular \(\varphi\in D(\mathbf\Delta,\Omega')\).
Suppose further that \(\mathbf\Delta\varphi\) is finite. Then \(\{\varphi<t\}\) is a set of locally finite perimeter in \(\Omega\) for a.e.\ \(t\in\setR\)
such that \(\{\varphi=t\}\cap\Omega\neq\emptyset\) and it holds that
\[
(\nabla\varphi\cdot\nu_{\{\varphi<t\}})_{\rm int}=(\nabla\varphi\cdot\nu_{\{\varphi<t\}})_{\rm ext}=-1\, ,\quad\Per(\{\varphi<t\},\cdot)\text{-a.e.\ in }\Omega\, .
\]
\end{proposition}

The primary focus of this note will be isoperimetric sets, that, as in the classical Riemannian setting, are much more regular than general sets of finite perimeter.

\begin{definition}\label{def:MinimizerCompactVariations}
Let $(X,\dist,\meas)$ be a metric measure space. We say that a subset $E\subset X$ is a \textit{volume constrained minimizer for compact variations in $X$} if whenever $F\subset X$ is such that $E\Delta F\subset K\Subset X$, and $\meas(K\cap E)=\meas(K\cap F)$, then $\Per(E)\leq \Per(F)$.

We say that a subset $E\subset X$, with $\meas(E)<\infty$, is an \textit{isoperimetric set} whenever for any $F\subset X$ with $\meas(F)=\meas(E)$ we have that $\Per(E)\leq \Per (F)$. 

Notice that an isoperimetric set in $X$ is a fortiori a volume constrained minimizer for compact variations in $X$.
\end{definition}

Let us recall a topological regularity result for volume constrained minimizers borrowed from \cite{AntonelliPasqualettoPozzetta21}. {A similar regularity result for local perimeter minimizers on PI spaces (without volume constraints) was obtained earlier in \cite{KinnunenKorteLorentShan}.}

\begin{theorem}[{\cite[Theorem 1.3 and Theorem 1.4]{AntonelliPasqualettoPozzetta21}}]\label{thm:RegularityIsoperimetricSets}
    Let $(X,\dist,\haus^N)$ be an $\RCD(K,N)$ space with $2\leq N<+\infty$ natural number, $K\in\mathbb R$. Let $E$ be a volume constrained minimizer for compact variations in $X$. Then $E^{(1)}$ is open, $\partial^eE=\partial E^{(1)}$, and $\partial E^{(1)}$ is locally uniformly $(N-1)$-Ahlfors regular in $X$. 
    
    Assume further there exists $v_0>0$ such that $\haus^N(B_1(x))\geq v_0$ for every $x\in X$, and that $E\subset X$ is an isoperimetric region. Then $E^{(1)}$ is in addition bounded, and $\partial E^{(1)}$ is $(N-1)$-Ahlfors regular in $X$. 
\end{theorem}

In the following, when $E$ is an isoperimetric region in a space $X$ as in \autoref{thm:RegularityIsoperimetricSets}, we will always assume that $E$ coincides with its open bounded representative given by $E^{(1)}$.


\subsection{Localization of the curvature-dimension condition}

We will rely on the so-called localization of the curvature-dimension condition. We give some basic background about it and address the reader to \cite{CavallettiMilmanCD,CM17,CM20} for a detailed account about this topic, under much more general assumptions.
\medskip

Let us consider an $\RCD(K,N)$ metric measure space $(X,\dist,\meas)$ for some $K\in\setR$ and $1<N<\infty$. Let $\Omega\subset X$ be an open subset and let $f:X\to\setR$ be the signed distance function from $\Omega$, i.e.,
\begin{equation}
    f(x):=\dist(x,\overline{\Omega})=\inf\{\dist(x,y)\, :\, y\in \Omega\}\, ,\quad\text{if $x\in X\setminus\Omega$}\, ,
\end{equation}
and
\begin{equation}
    f(x):=-\dist(x,X\setminus\Omega)=-\inf\{\dist(x,y)\, :\, y\in X\setminus\Omega\}\, ,\quad\text{if $x\in \Omega$}\, .
\end{equation}
The signed distance function $f$ induces an $\meas$-almost everywhere partition of $X$ into geodesics $X_{\alpha}$ indexed over a set $Q$. On a smooth Riemannian manifold, these geodesics would correspond to gradient flow lines of the signed distance function, or, equivalently, to integral curves of $-\nabla f$.\\
Rays $X_{\alpha}$ are often identified with intervals of the real line via the ray map $\gamma_{\alpha}:I_{\alpha}\to X_{\alpha}$, where $I_{\alpha}\subset \setR$ is an interval and $\gamma_{\alpha}$ is an isometry.

The almost-everywhere partition of $X$ into \emph{transport rays} induced by the signed distance function $f$ determines the following disintegration formula: 
 \begin{equation}\label{eq:disintegration2}
\meas=\int_{Q} h_{\alpha} \haus^{1} \res X_{\alpha}\,  {\mathfrak{q}}(\di \alpha)\,.
 \end{equation}
 The non-negative measure  $\mathfrak{q}$ in \eqref{eq:disintegration2}, defined on the set of indices $Q$, is obtained in a natural way from the essential partition $(X_{\alpha})_{\alpha \in Q}$ of $X$, roughly by projecting $\meas$  on the set $Q$ of equivalence classes (we refer to \cite{CM20} for the details).

The key property is that, if $(X,\dist,\meas)$ is an $\RCD(K,N)$ metric measure space, then each $h_{\alpha}$ is a $\CD(K,N)$ density over the ray $X_{\alpha}$ (see \cite[Theorem 3.6]{CM20}), i.e.,
\begin{equation}\label{eq:localprelA}
(\log h_{\alpha})''\leq - K - \frac{1}{N-1} \big( (\log h_{\alpha})' \big)^{2},
\end{equation}
in the sense of distributions and point-wise except countably many points, compare with \cite[Lemma A.3, Lemma A.5, Proposition A.10]{CavallettiMilmanCD}. 
Equivalently
\begin{equation}\label{eq:localprel}
\left(h_{\alpha}^{\frac{1}{N-1}}\right)''+\frac{K}{N-1}h_{\alpha}^{\frac{1}{N-1}}\le 0\, ,
\end{equation}
in the sense of distributions. This amounts to say that the curvature-dimension condition of the ambient space $(X,\dist,\meas)$ is inherited by the needles of the partition induced by $f$.

\medskip

With the help of the localization technique, we will be able to turn some estimates into one-dimensional comparison results for solutions of Riccati equations. We introduce here the relevant notation for our purposes.

Let us introduce the comparison functions
\begin{equation}
    s_{k,\lambda}(r):=\cos_k(r)-\lambda\sin_k(r)\, ,
\end{equation}
where 
\begin{equation}
   \cos_k''+k\cos_k=0\, ,\quad \cos_k(0)=1\, ,\quad \cos_k'(0)=0\, , 
\end{equation}
and 
\begin{equation}
    \sin_k''+k\sin_k=0\, ,\quad\sin_k(0)=0\, ,\quad \sin_k'(0)=1\, .
\end{equation}

Notice that $s_{k,-d}$ is a solution of
\begin{equation}
    v''+kv=0\, ,\quad v(0)=1\, ,\quad v'(0)=d\, .
\end{equation}
Moreover, $s_{0,\lambda}(r)=1-\lambda r$.

%
%
%
Let us fix $N>1$, $H\in\setR$ and $K\in\setR$. Then we introduce the Jacobian function 
\begin{equation}\label{eq:Jacobian}
\setR\ni r\mapsto J_{H,K,N}(r):=\left(\cos_{\frac{K}{N-1}}(r)+\frac{H}{N-1} \sin_{\frac{K}{N-1}}(r)\right)^{N-1}_+=\left(s_{\frac{K}{N-1},-\frac{H}{N-1}}(r)\right)^{N-1}_+\, .
\end{equation}

Notice that, when $K=0$ the expression for the Jacobian function simplifies into
\begin{equation}
\setR\ni r\mapsto J_{H,N}(r):=\left(1+\frac{H}{N-1} r\right)^{N-1}_+\, .
\end{equation}
We stress that the function $J_{H,K,N}$ is precisely the one involved in the one-dimensional comparison of $\CD(K,N)$ densities, see \cite[Corollary 4.3]{Ketterer20}.

\section{The distance function from isoperimetric sets}

Let $(M,g)$ be a smooth $N$-dimensional Riemannian manifold with Ricci curvature bounded from below by $K\in\setR$. Let $E\subset M$ be a set of finite perimeter which is an isoperimetric region. Then, by the classical regularity theory for constrained perimeter minimizers \cite{Almgrenreg,GonzalezMassariTamanini,morgan2003regularity}, $E$ has an open representative and $\partial E$ is smooth away from a set $\partial_sE$ of Hausdorff dimension $\dim_H(\partial_sE)\le N-8$. Moreover, $\partial E\setminus \partial_sE$ has constant mean curvature $c\in\setR$, in the classical sense and $\partial_sE$ can be characterized as the set of those points in $\partial E$ such that the tangent cone is not included in a half-space, thanks to \cite{Almgrenreg}.\\

The classical proof of the L\'evy--Gromov inequality for manifolds with positive Ricci curvature \cite[Appendix C]{Gromovmetric} combines the regularity results mentioned above with a Heintze--Karcher type estimate \cite{HeintzeKarcher}. In particular, the regularity theorem from \cite{Almgrenreg} is used in a crucial way to overcome the possible lack of smoothness of the isoperimetric boundary. The proof gives in particular the following result, valid for any lower curvature bound.

\begin{theorem}\label{thm:laplsmooth}
Let $(M,g)$ be a smooth $N$-dimensional Riemannian manifold with Ricci curvature bounded from below by $K\in\setR$ and let $E\subset X$ be an isoperimetric set.
Then, denoting by $f$ the signed distance function from $\overline{E}$ and by $c$ the value of the constant mean curvature of $\partial E\setminus \partial_sE$, it holds
\begin{equation}\label{eq:meanccompsmooth}
\boldsymbol{\Delta} f\ge  -(N-1)\frac{s'_{\frac{K}{N-1},\frac{c}{N-1}}\circ \left(-f\right)}{s_{\frac{K}{N-1},\frac{c}{N-1}}\circ \left(-f\right)}\, \, \,      \quad\text{on $E$ and }\quad
\boldsymbol{\Delta} f\le (N-1)\frac{s'_{\frac{K}{N-1},-\frac{c}{N-1}}\circ f}{s_{\frac{K}{N-1},-\frac{c}{N-1}}\circ f}    \quad\text{on $X\setminus \overline{E}$}\, .
\end{equation}
\end{theorem}

\begin{remark}\label{rm:cmc}
The very same conclusion above holds assuming that $E$ is a domain with smooth boundary and that $\partial E$ has constant mean curvature $c$. The proof of this variant does not require the deep regularity theorem from \cite{Almgrenreg}, but the assumptions are not natural for the applications to the isoperimetric problem when $N\ge 8$.
\end{remark}

Notice that the bounds in \eqref{eq:meanccompsmooth} make perfectly sense on a metric measure space $(X,\dist,\meas)$, even though most of the ingredients of the classical proof that we recalled above do not. 
\medskip

By \cite[Theorem 1.3]{AntonelliPasqualettoPozzetta21}, a volume constrained minimizer $E$ for compact variations enjoys analogous topological regularity properties of an isoperimetric region. More precisely, $E^{(1)}$ is open and $\partial^eE=\partial E^{(1)}$.  Hence we will always identify such a set with its representative $E^{(1)}$.

\begin{theorem}\label{thm:Isoperimetrici}
Let $(X,\dist,\haus^N)$ be an $\RCD(K,N)$ metric measure space for some $K\in\setR$ and $N\ge 2$ and let $E\subset X$ be a set of finite perimeter. Assume that $E$ is a volume constrained minimizer for compact variations in $X$.
Then, denoting by $f$ the signed distance function from $\overline{E}$, there exists $c\in\setR$ such that
\begin{equation}\label{eq:sharplaplacian}
\boldsymbol{\Delta} f\ge  -(N-1)\frac{s'_{\frac{K}{N-1},\frac{c}{N-1}}\circ \left(-f\right)}{s_{\frac{K}{N-1},\frac{c}{N-1}}\circ \left(-f\right)}      \quad\text{on $E$ and }\quad
\boldsymbol{\Delta} f\le (N-1)\frac{s'_{\frac{K}{N-1},-\frac{c}{N-1}}\circ f}{s_{\frac{K}{N-1},-\frac{c}{N-1}}\circ f}    \quad\text{on $X\setminus \overline{E}$}\, .
\end{equation}

\end{theorem}

\begin{remark}
The bounds in \eqref{eq:sharplaplacian} can be understood in the sense of distributions or in the sense of viscosity, see \autoref{thm:viscoimpldistri}. The two perspectives will be both relevant for the sake of the proof while in the applications it will be important to rely mostly on the distributional perspective.
\end{remark}

\begin{remark}
Notice that the comparison function in \eqref{eq:sharplaplacian} are not well defined globally on $\setR$, but only on a maximal interval $I_{K,N,c}\subset \setR$. In the course of the proof we will show that the signed distance function from $\overline{E}$ can attain only values in $I_{K,N,c}$ where the bounds in \eqref{eq:sharplaplacian} perfectly make sense.
\end{remark}

{ The proof of \autoref{thm:Isoperimetrici} is based on a careful adaptation of the argument in \cite[Theorem 5.2]{MoS21} that dealt with local perimeter minimizers without volume constraints.} We outline the strategy for smooth Riemannian manifolds, avoiding the technicalities and focusing on the case $K=0$.\\ 
We consider a weaker statement, corresponding to the limit of \eqref{eq:sharplaplacian} as $N\to\infty$. The self-improvement of the adimensional bound to the sharp dimensional bound is based on a classical computation with Jacobi fields away from the cut-locus of the distance function.\\
If this weaker statement fails, then there are points $x\in X\setminus E$ and $y\in E$ such that $\Delta f(x)>\Delta f(y)$ in a weak sense. This is already a subtle point, where we exploit the viscosity theory, since the distance function is not globally smooth. Thanks to a perturbation argument, we construct a smooth function $g$, touching $f$ from below only at $x$, and a smooth function $h$, touching $f$ from above only at $y$, such that $\Delta g(x)>\Delta h(y)$.\\
Then we introduce the functions
\begin{equation}\label{eq:dualheur}
   \bar{g}(z):=\sup_{w}\{g(w)-\dist(w,z)\}\, ,\quad\, \bar{h}(z):=\inf_{w}\{h(w)+\dist(z,w)\}\, . 
\end{equation}
Heuristically, $\bar{g}$ and $\bar{h}$ behave like distance functions from their respective level sets. Moreover, the non-linear transformation in \eqref{eq:dualheur} maintains the bounds on the Laplacian, if the ambient has non-negative Ricci curvature. Notice that a weak perspective on Laplacian bounds needs to be considered again, since the transformation does not preserve the regularity, in general. Neglecting the regularity issues, the conclusion would follow from \cite{AndrewsClut14}, which is based on a computation with Jacobi fields.
\\ 
The idea is then to slide the level sets of the functions $\bar{g}$ and $\bar{h}$ until they start crossing the isoperimetric set $E$. When this happens, we cut $\partial E$ along the level sets of $\bar{g}$ and $\bar{h}$ making sure to balance the interior and exterior perturbations so that, globally, the perturbation has the same volume of $E$. {Some care is needed in order to make sure that the two perturbations have disjoint supports.} Eventually, we estimate the perimeter of the perturbation and compare it with the perimeter of $E$, reaching a contradiction. 

Notice that we do not need to rely on the full regularity theory for isoperimetric sets on Riemannian manifolds. For the sake of the proof it is sufficient to know that the \emph{measure theoretic} boundary of the isoperimetric set is closed, which follows from the local density estimates obtained in \cite{AntonelliPasqualettoPozzetta21} (see also the previous \cite{KinnunenKorteLorentShan} for the case of local perimeter minimizers without volume constraints).

\begin{proof}
The proof will be divided into two steps. In the first one we are going to prove the weaker adimensional bounds
\begin{equation}\label{eq:adimensional laplacian comparison}
\boldsymbol{\Delta} f\ge c- Kf\,\quad\text{on $E$ and }\quad \boldsymbol{\Delta} f\le c - Kf\, \quad\text{on $X\setminus \overline{E}$}\, ,
\end{equation}
corresponding to the limit as $N\to \infty$ of the bounds in \eqref{eq:sharplaplacian}.\\
In the second part of the proof we will show how to obtain the sharp dimensional bounds with a by now standard application of the localization technique from \cite{CM17,CM20}.\\
{We remark that if \eqref{eq:adimensional laplacian comparison} (\eqref{eq:sharplaplacian} respectively) holds in a neighbourhood of $\partial E$, then \eqref{eq:adimensional laplacian comparison} (\eqref{eq:sharplaplacian} respectively) holds globally. This statement can be verified with the very same argument that we will discuss in Step 2 of the proof below. We omit the details, as this observation will not be needed for the rest of the proof.}

\medskip

\textbf{Step 1.} We will prove \eqref{eq:adimensional laplacian comparison} in the case $K=0$. {The modifications needed to address the case $K\neq 0$ will be discussed at the end of the step, see also Step 6 in the proof of \cite[Theorem 5.2]{MoS21}.}

Observe that, by \autoref{thm:viscoimpldistri}, {\eqref{eq:adimensional laplacian comparison} in the case $K=0$} is equivalent to the following claim.\\
\textbf{Claim}: the supremum of the values of the Laplacians of lower supporting functions of $f$ (as in \autoref{def:viscosity})
at touching points on $X\setminus\overline{E}$ is lower than the infimum of the values of Laplacians of upper supporting functions of $f$ at touching points on $E$. Indeed, if this statement holds, then letting $c$ be any value between the supremum and the infimum of the two sets, then $\Delta f\le c$ holds on $X\setminus \overline{E}$ in the viscosity sense and $\Delta f\ge c$ holds on $E$, again in the viscosity sense. By \autoref{thm:viscoimpldistri}, \eqref{eq:adimensional laplacian comparison} holds also in the sense of distributions.
\\
Observe that also the converse implication holds, again by the same kind of arguments. Since we will not need this implication we omit its proof.
\medskip

Let us prove the claim.\\
We argue by contradiction. If it is not true,
then we can find $x\in X\setminus \overline{E}$, $y\in E$, $\lambda>0$, $\delta>0$ (that we think to be very small, {in particular, $\lambda<f(x)$ and $\delta<\lambda$}) and supporting functions $\overline{\psi}:X\to\setR$ and $\overline{\chi}:X\to\setR$ with the following properties:
\begin{itemize}
\item[ia)] $\overline{\psi}:X\to\setR$ is Lipschitz and 
{it belongs to the domain of the measure-valued}
Laplacian on $B_{\lambda}(x)$;
\item[iia)] $\overline{\psi}(x)=f(x)$;
\item[iiia)] $\overline{\psi}(z)< f(z)$ for any $z\neq x$ and $\overline{\psi}<f-\delta$ on $X\setminus B_{\lambda}(x)$;
\end{itemize}
and 
\begin{itemize}
\item[ib)] $\overline{\chi}:X\to\setR$ is Lipschitz and 
{it belongs to the domain of the measure-valued}
Laplacian on $B_{\lambda}(y)$;
\item[iib)] $\overline{\chi}(y)=f(y)$;
\item[iiib)] $\overline{\chi}(z)>f(z)$ for any $z\neq y$ and $\overline{\chi}>f+\delta$ on $X\setminus B_{\lambda}(y)$.
\end{itemize}
Moreover, there exist $c\in\setR$ and $\eps>0$ such that
\begin{equation}\label{eq:visccon1}
{\boldsymbol{\Delta}}\overline{\psi}\ge c+\eps\, \quad\text{on $B_{\lambda}(x)$}\, 
\end{equation}
and 
\begin{equation}\label{eq:visccon2}
{\boldsymbol{\Delta}}\overline{\chi}\le c-\eps\, \quad\text{on $B_{\lambda}(y)$}\, .
\end{equation}
Notice that \autoref{thm:viscoimpldistri} yields {a priori only the} existence of {locally defined} functions $\psi$ and $\chi$ verifying the weak inequalities $\psi\le f$ and $\chi\ge f$ in place of the strict ones iiia) and iiib).\\ 
{In order to obtain the strict inequalities in iiia) and iiib) it is sufficient to subtract the function $\eps'\dist(x,\cdot)^2$ to $\psi$ and to sum the function $\eps'\dist(y,\cdot)^2$ to $\chi$, for $\eps'>0$ small enough. In this way we obtain new auxiliary functions $\hat{\psi}$ and $\hat{\chi}$. The fact that the inequalities \eqref{eq:visccon1} and \eqref{eq:visccon2} are not affected by this additive perturbation, up to slightly decreasing the value of $\eps$, follows from the Laplacian comparison theorem, see \cite{Gigli12}.\\
In order to extend the locally defined functions $\hat{\psi}$ and $\hat{\chi}$ to globally defined functions $\overline{\psi}$ and $\overline{\chi}$ while keeping their good properties, it is sufficient to employ a truncation argument. Namely, if $\hat{\psi}:B_{2\lambda}(x)\to\setR$ is such that $\hat{\psi}(x)=f(x)$, $\hat{\psi}(z)<f(z)$ for any $z\in B_{2\lambda}(x)$ with $z\neq x$, $\hat{\psi}<f-\delta$ on $B_{2\lambda}(x)\setminus B_{\lambda}(x)$ and $\boldsymbol{\Delta}\hat{\psi}\ge c+\eps$ on $B_{2\lambda}(x)$, we extend $\hat{\psi}$ to $-\infty$ on $X\setminus B_{2\lambda}(x)$ and set 
\begin{equation}\label{eq:truncation}
\overline{\psi}:=\max\{f-2\delta,\hat{\psi}\}\, .
\end{equation}
An analogous construction gives the sought global extension $\overline{\chi}$ of $\hat{\chi}$.
}


\medskip

We consider the transform of $\overline{\psi}$ through the Hopf--Lax duality and introduce $\phi:X\to\setR$ by letting
\begin{equation}\label{eq:HLphi}
\phi(z):=\sup_{w\in X}\{\overline{\psi}(w)-\dist(w,z)\}\, .
\end{equation}
Analogously, we let $\eta:X\to\setR$ be defined by
\begin{equation}\label{eq:HLeta}
\eta(z):=\inf_{w\in X}\{\overline{\chi}(w)+\dist(z,w)\}\, .
\end{equation}
\medskip

Let $X_{\Sigma}$ and $Y_{\Sigma}$ be the sets of touching points of minimizing geodesics from $x$ and $y$ respectively to $\Sigma:=\partial E$, i.e.
\begin{equation}
X_{\Sigma}:=\{w\in \partial E\, :\, f(x)-f(w)=\dist(x,w)\},
\end{equation}
and
\begin{equation}
Y_{\Sigma}:=\{w\in \partial E\, :\, f(w)-f(y)=\dist(y,w)\}\, .
\end{equation}
It is easy to verify that $X_{\Sigma}$ and $Y_{\Sigma}$ are compact subsets of $\partial E$. 

{It is elementary to check that} $\phi\le f$, $\eta\ge f$ {because $\overline{\psi}\le f$ and $\overline{\chi}\ge f$ respectively}. Moreover, both $\phi$ and $\eta$ are $1$-Lipschitz functions, {because they are defined as suprema and infima of families of $1$-Lipschitz functions and they are finite at some point.}\\ 
Moreover, there exist neighbourhoods of $U_{\Sigma}\supset X_{\Sigma}$ and $V_{\Sigma}\supset Y_{\Sigma}$ such that: 
\begin{itemize}
\item[a)] $\abs{\nabla\phi}=1$ holds $\haus^N$-a.e. on $U_{\Sigma}$; 
\item[b)] $\abs{\nabla\eta}=1$ holds $\haus^N$-a.e. on $V_{\Sigma}$;
\item[c)] $\boldsymbol{\Delta}\phi \ge c+\eps$ on $U_{\Sigma}$;
\item[d)] $\boldsymbol{\Delta}\eta\le c-\eps$ on $V_{\Sigma}$;
\item[e)] $\phi(z)=f(z)$ for any $z\in X$ such that $f(x)-f(z)=\dist(x,z)$; 
\item[f)]$\eta(z)=f(z)$ for any $z\in X$ such that $f(z)-f(y)=\dist(z,y)$.
\end{itemize}
{We check the claims relative to $\phi$, the verification of the claims relative to $\eta$ being completely analogous.\\
The proof is the same as the proof of the analogous claims in Step 3 of the proof of \cite[Theorem 5.2]{MoS21}. We repeat the argument below for the sake of readability.\\
Let $x_E\in X_{\Sigma}\subset \partial E$ be any footpoint of minimizing geodesic from $x$ to $\overline{E}$. In particular, $f(x_E)=0$ and $f(x)-f(x_E)=\dist(x,x_E)$. Let $\gamma:[0,\dist(x,x_E)]\to X$ be a unit speed minimizing geodesic between $\gamma(0)=x_E$ and $\gamma(\dist(x,x_E))=x$. Observe that 
\begin{equation}\label{eq:fgammat}
f(\gamma(t))=t\, \quad \text{for any $t\in[0,\dist(x,x_E)]$}\, . 
\end{equation}
Moreover, 
\begin{equation}\label{eq:phigammat}
\phi(\gamma(t))=f(\gamma(t)), \quad \text{for any $t\in [0,\dist(x,x_E)]$},
\end{equation}
 and, for any such $t$, the supremum defining $\phi(\gamma(t))$ in \eqref{eq:HLphi} is attained only at $x$.\\
Indeed, by iiia) above, $\overline{\psi}<f-\delta$ on $X\setminus B_{\lambda}(x)$. Hence, for any $z\in X$ such that $\phi(z)>f(z)-\delta$, we can restrict the supremum defining $\phi(z)$ in \eqref{eq:HLphi} to $\overline{B_{\lambda}(x)}$. Since $\overline{B_{\lambda}(x)}$ is compact, the supremum is attained. In details, if $\phi(z)>f(z)-\delta$, then 
\begin{equation}\label{eq:phizlefz}
\phi(z)= \sup_{y\in \overline{B_{\lambda}(x)}}\{\overline{\psi}(y)-\dist(y,z)\}=\overline{\psi}(y_z)-\dist(y_z,z)\le f(y_z)-\dist(y_z,z)\le f(z)\, ,
\end{equation}
for some $y_z\in \overline{B_{\lambda}(x)}$. In particular, whenever $\phi(z)=f(z)$, all the inequalities above become equalities. Hence $\overline{\psi}(y_z)=f(y_z)$, that implies $y_z=x$ by iia) and iiia), and $f(z)-f(x)=-\dist(x,z)$. Viceversa, if $f(z)-f(x)=-\dist(x,z)$ then $\phi(z)=f(z)$ and the supremum defining $\phi(z)$ is attained (only) at $x$. In particular, these observations prove e).
\medskip

We claim that
\begin{equation}\label{eq:nablaphi1}
\abs{\nabla \phi}=1, \quad \haus^N\text{-a.e. on } \{\phi>f-\delta\}\setminus B_{\lambda}(x)\, ,
\end{equation} 
that is clearly enough to prove a).\\
In order to verify this claim, we let $z\in \{\phi>f-\delta\}\setminus B_{\lambda}(x)$. By the argument above, the supremum defining $\phi(z)$ is a maximum and it is attained at some $x_z\in \overline{B_{\lambda}(x)}$. By assumption $x_z\neq z$. Let us consider a minimizing geodesic $\gamma:[0,\dist(z,x_z)]\to X$ with unit speed connecting $z$ with $x_z$. We claim that
\begin{equation}\label{eq:slope1}
\phi(\gamma(t))=\phi(z)+t\, ,\quad\text{for any $t\in[0,\dist(z,x_z)]$}\, .
\end{equation}
The inequality $\phi(\gamma(t))\le \phi(z)+t$ follows from the fact that $\phi$ is $1$-Lipschitz. We only need to prove that $\phi(\gamma(t))\ge\phi(z)+t$. To this aim,  observe that
\begin{align*}
\phi(\gamma(t))&=\sup_{y\in X}\{\overline{\psi}(y)-\dist(y,\gamma(t))\}\\
&\ge \overline{\psi}(x_z)-\dist(\gamma(t),x_z)\\
&=\overline{\psi}(x_z)-\dist(z,x_z)+t\\
&=\phi(z)+t\, .
\end{align*}
From \eqref{eq:slope1} we infer that, for any $z\in \{\phi>f-\delta\}\setminus B_{\lambda}(x)$, the function $\phi$ has slope $1$ at $z$. The conclusion that $\abs{\nabla\phi}=1$-a.e. on $\{\phi>f-\delta\}\setminus B_{\lambda}(x)$ follows from the classical a.e. identification between slope and upper gradient obtained in \cite{Cheeger99}.

We are left to prove the Laplacian bound c). By construction, $\overline{\psi}$ verifies the Laplacian bound \eqref{eq:visccon1} on $B_{\lambda}(x)$. We already observed that for points $z\in \{\phi>f-\delta\}\setminus B_{\lambda}(x)$ the supremum defining $\phi(z)$ is a maximum attained in $\overline{B_{\lambda}(x)}$, hence we obtain by \autoref{thm:HLlapla} (more precisely, by the dual version with infima replaced by suprema) that
\begin{equation}\label{eq:Deltaphige}
\boldsymbol{\Delta}\phi\ge c+\eps\, \quad\text{on $\{\phi>f-\delta\} \setminus B_{\lambda}(x)$}\, ,
\end{equation} 
in the sense of distributions. The set $\{\phi>f-\delta\} \setminus B_{\lambda}(x)$ is easily seen to be a neighbourhood of $X_{\Sigma}$ for $\lambda$ small enough, as $\phi=f$ on $X_{\Sigma}$, hence we have proved c).

\medskip

Our next goal is to reduce to the case where $X_{\Sigma}=\{x_E\}$ and $Y_{\Sigma}=\{y_E\}$ are singletons. We discuss the reduction for $X_{\Sigma}$, the case of $Y_{\Sigma}$ being completely analogous.\\

As above, we let $x_E\in X_{\Sigma}\subset \partial E$ be any footpoint of minimizing geodesic from $x$ to $\overline{E}$. In particular, $f(x_E)=0$ and $f(x)-f(x_E)=\dist(x,x_E)$. Let $\gamma:[0,\dist(x,x_E)]\to X$ be a unit speed minimizing geodesic between $\gamma(0)=x_E$ and $\gamma(\dist(x,x_E))=x$.\\
By the non-branching property for minimizing geodesics in $\RCD(K,N)$ spaces, see \cite[Theorem 1.3]{DengNonBranching}, for any $t\in [0,\dist(x,x_E))$ the minimizing geodesic from $\gamma(t)$ to $E$ is unique, and it coincides with the restriction of $\gamma$ to the interval $[0,t]$. In particular, the footpoint of the minimizing geodesic from $\gamma(t)$ to $E$ is unique and coincides with $x_E$.\\
Moreover, we can substitute any such point $\gamma(t)$ for $0<t<\dist(x,x_E)-\lambda$ to $x$ in the contradiction argument. Indeed, the function $\phi$ satisfies the following properties:
\begin{itemize}
    \item[i)] it is Lipschitz and it belongs locally to the domain of the measure-valued Laplacian. The second statement has been already verified in $\{\phi>f-\delta\} \setminus B_{\lambda}(x)$ and it holds globally by the very definition of $\phi$ and the Laplacian comparison theorem;
    \item[ii)] $\phi\le f$ and $\phi(\gamma(t))=f(\gamma(t))$;
    \item[iii)] $\boldsymbol{\Delta}\phi\ge c+\eps$ in a neighbourhood of $\gamma(t)$. 
\end{itemize}
With a perturbation and a truncation argument completely analogous to the one employed before \eqref{eq:truncation}, we can modify $\phi$ and assume that the inequality in ii) is strict away from $\gamma(t)$ and uniformly strict away from a small ball centred at $\gamma(t)$.

The effect of this reduction is that we can assume that the original points $x$, $y$ in the contradiction argument are as close as we wish to $\partial E$. Moreover they have unique minimizing geodesics to $\partial E$, hence in particular unique footpoints on $\partial E$, that we shall denote by $x_E$ and $y_E$ respectively.\\
We will not rename the points $x$, $y$ nor the auxiliary functions $\phi$ and $\eta$, for the ease of notation.

}

\medskip


We need to consider two cases. {The case $x_E=y_E$ and the case $x_E\neq y_E$.}
\medskip

\textbf{Case 1: {$x_E=y_E$.}}

By construction, it holds $\eta\ge f\ge\phi$. {Moreover, $\eta(x_E)=f(x_E)=\phi(x_E)$, by e), f).\\
Let us set $g:=\eta-\phi$. Observe that $g\ge 0$ and $g(x_E)=0$. Moreover, by c) and d), there exists a neighbourhood of $x_E$ where}
\begin{equation}
\boldsymbol{\Delta} g\le (c-\eps)+(-c-\eps)\le -2\eps\, .
\end{equation} 
We get a contradiction, since $g$ would be a non-constant superharmonic function attaining its minimum at an interior point, see \cite[Theorem 2.8]{GigliRigoni}.

\medskip

\textbf{Case 2: {$x_E\neq y_E$}.}

{

Let us start by proving that for small values of $s\in(-\delta,0)$, we can cut $E$ along a level set of $\phi$ to obtain inner perturbations $E_{s,0}\subset E$, compactly supported on suitable balls of arbitrary small radius. The argument is analogous to Step 4 in the proof of \cite[Theorem 5.2]{MoS21} and we repeat it here for the sake of readability.

Let us define
\begin{equation*}
E_{s,0}:=E\setminus \{\phi > s\}\, .
\end{equation*}
Observe that for $s=0$ it holds $\{\phi>0\}\cap E=\emptyset$, since $\{\phi>0\}\subset \{f>0\}\subset X\setminus E$ by construction. When we decrease the value of $s$, the super-level set $ \{\phi >s\}$ starts cutting $E$.\\ 
Recall that $x_E\in \partial E$ is the footpoint of the minimizing geodesic from $x$ to $\overline{E}$. We claim that for any $s<0$ sufficiently close to $0$, $E_{s,0}$ is a perturbation of $E$ supported in a small ball $B_{r}(x_E)$, i.e. $\{\phi>s\}\cap E\subset B_r(x_E)$. To prove this claim, it is enough to observe that  from $f\le 0$ on $E$, $\phi(x_E)=0$,  and $B_{\lambda}(x)\subset X\setminus \bar{E}$,  we get
\begin{equation}\label{eq:phi>tinclusion}
\{\phi>s\}\cap E \subset \{\phi>f-\delta\}\setminus \overline{B_{\lambda}(x)}  \quad \text{for any } s\in (-\delta,0)\, .
\end{equation}
Moreover, for every $z\in \{\phi>s\}\cap E$,  the maximum defining $\phi(z)$ is attained inside $\overline{B_{\lambda}(x)}$, see \eqref{eq:phizlefz} and the nearby discussion.\\ 
Now we wish to bound the distance from $x_E$ to any point in $\{\phi>s\}\cap E$. For any $z\in\{\phi>s\}\cap E$, there exists $x_z\in \overline{B_{\lambda}(x)}$ such that
\begin{equation*}
\phi(z)=\overline{\psi}(x_z)-\dist(x_z,z)\le f(x_z)-\dist(x_z,z)\le \lambda+\dist(x,\overline{E})-\dist(x_z,z)\, .
\end{equation*}
Hence 
\begin{equation}
\dist(x_z,z)\le \dist(x,\overline{E})+\lambda-\phi(z)\le \dist(x,\overline{E})+\lambda-s\, .
\end{equation}
In particular, we can bound the distance of any point in $\{\phi>s\}\cap E$ from $x$, and hence from $x_E$, and obtain
\begin{equation}\label{eq:distEtxE}
\{\phi>s\}\cap E \subset B_{r}(x_E), \quad r:= 2 \dist(x, \overline{E})+2\lambda+\delta\, .
\end{equation}
The effect of this construction is that for $x$ close enough to $\overline{E}$, and $\lambda, \delta>0$ sufficiently small, $r:= 2 \dist(x, \overline{E})+2\lambda+\delta$ is arbitrarily small. It follows that, for every $r>0$, one can perform the above construction in order to obtain $x_E\in \partial E$  and a family of  inner perturbations $(E_{s,0})_{s\in (-\delta, 0)}$ of $E$,  so that $E\setminus E_{s,0}\subset B_{r}(x_E)$.\\
A completely analogous verification shows that, for $0<t<\delta$, the set $E_{0,t}:=E\cup \{\eta\leq t\}$ is a perturbation of $E$ compactly supported in a small ball $B_r(y_E)$.\\
When $r<\dist(x_E,y_E)/2$, the interior and the exterior perturbations have disjoint supports. It is also elementary to check that the two perturbations are non-trivial.
\medskip
}


Let us set now, for any $-\delta\le s\le  0<t<\delta$,
\begin{equation}
E_{s,t}:=E\setminus \{\phi\ge s\}\cup\{\eta\le t\}\, .
\end{equation}
We claim that there exist values $s,t$ in the range above such that 
\begin{equation}\label{eq:volpres}
\haus^N(E_{s,t})=\haus^N(E)
\end{equation}
and 
\begin{equation}\label{eq:perturb}
\haus^N(E_{s,0})<\haus^N(E)<\haus^N(E_{0,t})\, .
\end{equation}
In order to establish the claim it is sufficient to prove that 
\begin{equation}
(s,t)\mapsto \haus^N(E_{s,t}),
\end{equation}
is a continuous function. {Indeed, \eqref{eq:perturb} follows immediately from the non-triviality of the perturbations.} The sought continuity is a direct consequence of the $1$-Lipschitz regularity of $\phi$ and $\eta$, together with the compactness of the perturbations $E_{s,t}\Delta E$ and the properties $\abs{\nabla\phi}=\abs{\nabla\eta}=1$-a.e., which guarantee that 
\begin{equation}
\haus^N\left(\{\phi=s\}\cap U_{\Sigma}\right)=\haus^N\left(\{\eta=t\}\cap V_{\Sigma}\right)=0\, ,
\end{equation}
for any $-\delta\le s\le 0<t<\delta$.

Given the claim, it is easy to find by a continuity argument $s$ and $t$ such that \eqref{eq:volpres} and \eqref{eq:perturb} hold. Moreover
letting $\Omega$ be the open neighbourhood of $\{x_E,y_E\}$ where the perturbation $E_{s,t}\Delta E$ is compactly contained, it holds, by \cite[Proposition 6.1]{BPSGaussGreen}, see \autoref{prop:6.1BPS},
\begin{equation}
\left(\nabla\phi\cdot\nu_{\{\phi<s\}}\right)_{\mathrm{int}}=\left(\nabla\phi\cdot\nu_{\{\phi<s\}}\right)_{\mathrm{ext}}=-1\, \quad \Per_{\{\phi<s\}}\text{-a.e. on $\Omega$},
\end{equation}
and 
\begin{equation}
\left(\nabla\eta\cdot\nu_{\{\eta<t\}}\right)_{\mathrm{int}}=\left(\nabla\eta\cdot\nu_{\{\eta<t\}}\right)_{\mathrm{ext}}=-1\, \quad \Per_{\{\eta<t\}}\text{-a.e. on $\Omega$}\, .
\end{equation}
Notice that in order to apply \cite[Proposition 6.1]{BPSGaussGreen} it is necessary to multiply $\phi$ and $\eta$ by regular cut-off functions compactly supported into $\Omega$. This can be easily done thanks to the existence of regular cut-off functions on $\RCD(K,N)$ spaces, see \cite[Lemma 6.7]{AmbrosioMondinoSavareJGA} and \cite{MondinoNaber,Gigli14}. In the rest of the proof we will assume that $\phi$ and $\eta$ have compact supports, without changing the notation.

\medskip

We are going to reach a contradiction comparing the perimeter of $E_{s,t}$ with that of $E$, arguing as in the final part of the proof of \cite[Theorem 5.2]{MoS21}. We estimate separately the differences in the perimeter coming from the perturbation induced by $\phi$ and $\eta$ by disjointedness:
\begin{equation}\label{eq:sepcontributions}
\Per(E_{s,t})-\Per(E)= \left(\Per(E_{(s,0)},\Omega)-\Per(E,\Omega)\right)+\left(\Per(E_{(0,t)},\Omega)-\Per(E,\Omega)\right)\, .
\end{equation}

The two contributions can be estimated as follows. Let us set $F:=E\cap \{\phi>s\}$ and $G:=\{\eta<t\}\setminus E$. Observe that by \eqref{eq:volpres} it holds $\haus^N(F)=\haus^N(G)$.

On the one hand we can apply \cite[Theorem 5.2]{BPSGaussGreen}, see \autoref{thm:GaussGreen}, with the sharp trace estimates, with test function identically equal to $1$, vector field $V=\nabla \phi$ and set of finite perimeter $F$. We obtain
\begin{align}
\nonumber \int_{F^{(1)}}\boldsymbol{\Delta}\phi&=-\int_{\mathcal{F}F}\left(\nabla\phi\cdot\nu_F\right)_{\mathrm{int}}\di\Per\\
\nonumber &\leq-\int_{E^{(1)}\cap\mathcal{F}\{\phi>s\}}(\nabla \phi\cdot \nu_{\{\phi>s\}})_{\mathrm{int}}\di\Per-\int_{\mathcal{F}E\cap \{\phi>s\}^{(1)}}\left(\nabla\phi\cdot\nu_E\right)_{\mathrm{int}}\di\Per\\
\nonumber &=-\haus^{N-1}\left(E^{(1)}\cap\mathcal{F}\{\phi>s\}\right)-\int_{\mathcal{F}E\cap \{\phi>s\}^{(1)}}\left(\nabla\phi\cdot\nu_E\right)_{\mathrm{int}}\di\Per\\
&\le -\haus^{N-1}\left(E^{(1)}\cap\mathcal{F}\{\phi>s\}\right)+\haus^{N-1}\left(\mathcal{F}E\cap \{\phi>s\}^{(1)}\right)\, ,\label{eq:boundgap1}
\end{align}
where the first equality follows from \cite[Theorem 5.2]{BPSGaussGreen}, the first inequality follows from \cite[Proposition 5.4]{BPSGaussGreen}, see \autoref{prop:5.4BPS}, and the fact that $$\mathcal{F}(E\cap\{\phi>s\})\sim \left(E^{(1)}\cap\mathcal{F}(\phi>s)\right)\sqcup \left(\mathcal{F}E\cap\{\phi>s\}^{(1)}\right)\sqcup\left(E^{(1/2)}\cap F^{(1/2)}\right),$$ as a consequence of \cite[Proposition 4.2]{BPSGaussGreen}, while the last inequality follows from the sharp trace bound $\abs{\left(\nabla\phi\cdot\nu_E\right)_{\mathrm{int}}}\le 1$ in \cite[Theorem 5.2]{BPSGaussGreen}, see \autoref{thm:GaussGreen}.

The analogous computation with $\nabla\eta$ in place of $\nabla \phi$ and $G$ in place of $F$ yields to
\begin{equation}\label{eq:boundgap2}
\int_{G^{(1)}}\boldsymbol{\Delta}\eta\ge \haus^{N-1}\left(E^{(0)}\cap \mathcal{F}\{\eta<t\}\right)-\haus^{N-1}\left(\mathcal{F}E\cap \{\eta<t\}^{(1)}\right)\, .
\end{equation}
\medskip

Now, the bounds on $\boldsymbol{\Delta}\phi$ and $\boldsymbol{\Delta}\eta$ imply 
\begin{equation}\label{eq:boundlap1}
\int_{F^{(1)}}\boldsymbol{\Delta}\phi\ge (c+\eps)\haus^N(F)\,
\end{equation}
and
\begin{equation}\label{eq:boundlap2}
\int_{G^{(1)}}\boldsymbol{\Delta}\eta\le (c-\eps)\haus^N(G)\, .
\end{equation}
Hence, by \eqref{eq:sepcontributions}, \eqref{eq:boundgap1}, \eqref{eq:boundgap2}, \eqref{eq:boundlap1} and \eqref{eq:boundlap2}
\begin{align*}
\Per(E_{s,t})-\Per(E)=& \left(\Per(E_{(s,0)},\Omega)-\Per(E,\Omega)\right)+\left(\Per(E_{(0,t)},\Omega)-\Per(E,\Omega)\right)\\
= &\left(\haus^{N-1}\left(E^{(1)}\cap\mathcal{F}\{\phi>s\}\right)- \haus^{N-1}\left(\mathcal{F}E\cap \{\phi>s\}^{(1)}\right)\right)\\
&+  \left(\haus^{N-1}\left(E^{(0)}\cap \mathcal{F}\{\eta<t\}\right)-\haus^{N-1}\left(\mathcal{F}E\cap \{\eta<t\}^{(1)}\right)\right)\\
\le &\int_{G^{(1)}}\boldsymbol{\Delta}\eta-\int_{F^{(1)}}\boldsymbol{\Delta}\phi\\
\le& (c-\eps)\haus^N(G)-(c+\eps)\haus^N(F)=-2\eps\haus^N(F)<0\, ,
\end{align*}
yielding to the sought contradiction.
\medskip

{
In order to cover the case of a general lower Ricci curvature bound $K\in\setR$ we can modify the argument above as follows.\\
In the contradiction argument at the very beginning of the proof we obtain functions $\bar{\psi}$ and $\bar{\chi}$ satisfying the same properties ia)-iiia) and ib)-iiib). The bounds \eqref{eq:visccon1} and \eqref{eq:visccon2} are now replaced by the bounds
\begin{equation}
\boldsymbol{\Delta}\overline{\psi}\ge c+\eps -Kf \quad \text{on $B_{\lambda}(x)$},
\end{equation}
and 
\begin{equation}
\boldsymbol{\Delta}\overline{\chi}\le c-\eps -Kf\, ,\quad \text{on $B_{\lambda}(y)$}\, ,
\end{equation}
respectively.\\
Arguing in the very same way as we did above, by employing \autoref{thm:HLlapla} and taking into account the correction terms due to the general lower Ricci curvature bound $K$, the functions $\phi$ and $\eta$ satisfy the Laplacian bounds $\boldsymbol{\Delta}\phi\ge c+\eps/2$ and $\boldsymbol{\Delta}\eta\le c-\eps/2$ in a neighbourhood of $X_{\Sigma}$ and $Y_{\Sigma}$ respectively. The rest of the argument carries over as in the $K=0$ case.
}

\medskip

\textbf{Step 2.}
Let us see how to pass from the adimensional estimates in \eqref{eq:adimensional laplacian comparison} to the sharp Laplacian comparison in \eqref{eq:sharplaplacian}.

We will rely on the localization technique from \cite{CM17,CM20}, following the proof of \cite[Theorem 5.2]{MoS21}. Let us prove the sharp Laplacian comparison on $X\setminus E$, the bound in the interior follows from an analogous argument.\\
From  \cite[Corollary 4.16]{CM20}, we know that 
\[\boldsymbol{\Delta} f \res X\setminus \overline{E}= (\boldsymbol{\Delta} f)^{\text{reg}}\res X\setminus \overline{E}+ (\boldsymbol{\Delta} f)^{\text{sing}}\res X\setminus \overline{E}\, ,\]
 where the singular part $(\boldsymbol{\Delta} f)^{\text{sing}}\perp \haus^{N}$ satisfies $(\boldsymbol{\Delta} f)^{\text{sing}} \res X\setminus \overline{E}\leq 0$ and the regular part $(\boldsymbol{\Delta} f)^{\text{reg}} \ll \haus^{N}$ admits the representation formula
\begin{equation}\label{eq:repDeltafreg}
(\boldsymbol{\Delta} f)^{\text{reg}} \res X\setminus \overline{E} = (\log h_{\alpha})' \haus^{N} \res X\setminus \overline{E}\,.
\end{equation}
In \eqref{eq:repDeltafreg}, $Q$ is a suitable set of indices,  $(h_{\alpha})_{\alpha\in Q}$  are suitable densities defined on geodesics $(X_{\alpha})_{\alpha \in Q}$, which are essentially partitioning $X\setminus \overline{E}$ (in the smooth setting, $(X_{\alpha})_{\alpha\in Q}$ correspond to the integral curves of $\nabla \dist_{E}$; note that here we are using the reverse parameterization of $X_{\alpha}$ with respect to \cite{CM20}, hence the reversed sign in the right hand side of \eqref{eq:repDeltafreg}), such that the following disintegration formula holds: 
 \begin{equation}\label{eq:disintegration}
 \haus^{N}\res X\setminus \overline{E}=\int_{Q} h_{\alpha} \haus^{1} \res X_{\alpha}\,  {\mathfrak{q}}(\di \alpha)\,.
 \end{equation}
 The non-negative measure  $\mathfrak{q}$ in \eqref{eq:disintegration}, defined on the set of indices $Q$, is obtained in a natural way from the essential partition $(X_{\alpha})_{\alpha \in Q}$ of $X\setminus \overline{E}$, roughly by projecting $\haus^{N}\res X\setminus \overline{E}$  on the set $Q$ of equivalence classes (we refer to \cite{CM20} for the details).
 \\The key point for the proof of Step 2 is that each $h_{\alpha}$ is a $\CD(K,N)$ density over the ray $X_{\alpha}$ (see \cite[Theorem 3.6]{CM20}), i.e.,
\begin{equation}\label{eq:DiffEqCDKN}
(\log h_{\alpha})''\leq - K - \frac{1}{N-1} \big( (\log h_{\alpha})' \big)^{2},
\end{equation}
in the sense of distributions and point-wise except countably many points, compare with \cite[Lemma A.3, Lemma A.5, Proposition A.10]{CavallettiMilmanCD}. 
Equivalently
\begin{equation}\label{eq:halphaCD}
\left(h_{\alpha}^{\frac{1}{N-1}}\right)''+\frac{K}{N-1}h_{\alpha}^{\frac{1}{N-1}}\le 0\, ,
\end{equation}
in the sense of distributions.
Moreover, from \eqref{eq:repDeltafreg} and \eqref{eq:adimensional laplacian comparison} we know that 
\begin{equation}\label{eq:halpha'Kdist}
(\log h_{\alpha})'(\dist_{\overline{E}}) \leq c- K \, \dist_{\overline{E}} \, \text{ a.e. on $X_{\alpha}$, for $\mathfrak{q}$-a.e. $\alpha\in Q$}\, .
\end{equation}

The sharp estimates in \eqref{eq:sharplaplacian} then follow from the standard Riccati comparison (see for instance \cite[Lemma 4.10]{Burtscheretal20}) applied to the functions $v(r):=\left(h_{\alpha}(r)/h_{\alpha}(0)\right)^{\frac{1}{N-1}}$ which verify
\begin{equation}
v''+\frac{K}{N-1}v\le 0\, ,    
\end{equation}
in the sense of distributions by \eqref{eq:halphaCD}, $v(0)=1$ and $v'(0)\le c/(N-1)$ by \eqref{eq:halpha'Kdist}. 
\end{proof}



\begin{definition}[Mean curvature barriers for isoperimetric sets]
Let $(X,\dist,\haus^N)$ be an $\RCD(K,N)$ space, let $E\subset X$ be an isoperimetric set. We call any constant $c$ such that \eqref{eq:sharplaplacian} holds a \emph{mean curvature barrier} for $\partial E$. 
\end{definition}
For discussions concerning the uniqueness of $c$ as in the previous definition, and comparison with the Riemannian setting, we refer the reader to the forthcoming remarks.

\begin{remark}\label{rem:UniquenessC}
If $(M^n,g)$ is a smooth Riemannian manifold (with Ricci curvature uniformly bounded from below) and $E\subset M$ is an isoperimetric set, then the constant $c$ obtained via \autoref{thm:Isoperimetrici} is unique and equal to the constant mean curvature of the regular part of $\partial E$.\\ 
The validity of a similar statement for isoperimetric sets in $\RCD(K,N)$ metric measure spaces $(X,\dist,\haus^N)$ goes beyond the scope of this note and is left to the future investigation.
\end{remark}

\begin{remark}
Let $(D,\dist,\haus^2)$ be a two dimensional flat disk with canonical metric measure structure and boundary $\partial D$. Let $(K,\dist_K,\haus^2)$ be the doubling of $D$ along its boundary $\partial D$ and set $k>0$ the curvature of $\partial D$. It is a classical fact that $(K,\dist_K,\haus^2)$ is an Alexandrov space with non-negative curvature, in particular it is an $\RCD(0,2)$ space, but it is not a smooth Riemannian manifold. Observe that each of the two isometric copies of $D$ inside $K$ verifies the bound \eqref{eq:sharplaplacian} for any $c\in[-k,k]$ (for $K=0$ and $N=2$). { Even though $D\subset K$ is not an isoperimetric set, cf.\ with \cite[Theorem 5.4]{CFGNSY05},} this example illustrates that the uniqueness of the mean curvature barrier is a delicate issue in the non-smooth setting.
\end{remark}

The following mild regularity results are obtained arguing verbatim as in \cite[Proposition 5.4, Theorem 5.5, Proposition 6.14]{MoS21}. This can be done since \autoref{thm:Isoperimetrici} is the counterpart  of \cite[Theorem 5.2]{MoS21} for isoperimetric sets, while \cite[Lemma 6.12]{MoS21} holds for volume constrained minimizers for compact variations as well, and \cite[Lemma 2.42]{MoS21} holds for volume constrained minimizers for compact variations since they are quasiminimal sets according to \cite[Theorem 3.24]{AntonelliPasqualettoPozzetta21}.

\begin{proposition}\label{prop:laplfull}
Let $(X,\dist,\haus^N)$ be an $\RCD(K,N)$ space for some $K\in\setR$ and $N\ge 2$. Let $E\subset X$ be a volume constrained minimizer for compact variations in $X$. {Let $\dist_{\overline E}$ be the distance function from the set $\overline E$, and let $\dist^s_E$ be the signed distance function from $E$, with the convention that it is positive outside $E$ and negative inside $E$.}

Then $\dist_{\overline E}$ and $\dist_E^s$ have locally measure valued Laplacian in $X$ and the following hold
\begin{equation}
    \begin{split}
        \boldsymbol{\Delta} \dist_{\overline E} &= \haus^{N-1}\res\partial E+\boldsymbol{\Delta} \dist_{\overline E}\res (X\setminus\overline E), \\
        \boldsymbol{\Delta} \dist_E^s\res\partial E&=0.
    \end{split}
\end{equation}
\end{proposition}

\begin{proof}
We only provide an indication of the strategy of the proof, that can be obtained with minor modifications with respect to the case of local perimeter minimizers considered in \cite{MoS21}.

The proof is divided into two steps: the verification that $\dist_{\overline E}$ and $\dist^s_E$ admit locally measure valued Laplacian and the computation of the singular part of their Laplacian along $\partial E$. 
\medskip

In order to prove that $\dist_{\overline E}$ and $\dist^s_E$ admit locally measure valued Laplacian it is sufficient to uniformly bound the volumes of the $t$-tubular neighbourhoods of $\partial E$ as $Ct$ when $t\to 0$ for some constant $C>0$ and to pass to the limit in the Gauss-Green integration by parts formulae on super-level sets $\{\dist_{\overline E}>t_i\}$ for suitably chosen sequences $t_i\downarrow 0$.\\
The uniform volume bound for the tubular neighbourhoods of $\partial E$ follows from \cite[Theorem 3.24]{AntonelliPasqualettoPozzetta21}, where quasiminimality of isoperimetric sets is proved, and \cite[Lemma 2.42]{MoS21}. The conclusion follows arguing as in Step 1 of the proof of \cite[Proposition 5.4]{MoS21} (see also the previous \cite{BrueNaberSemola20}).
\medskip

Thanks to {the arguments in the proof of \cite[Theorem 7.4]{BrueNaberSemola20}}, $\boldsymbol{\Delta}\dist_{\overline E}\ll \haus^{N-1}$ and $\boldsymbol{\Delta} \dist^s_{E}\ll \haus^{N-1}$. To conclude, it suffices to compute the densities of these measures with respect to $\haus^{N-1}\res\partial E$, which is a locally doubling finite measure. This can be done with a classical blow-up argument, as in Step 2 of the proof of \cite[Proposition 5.4]{MoS21}. In order to prove that at regular points of $\partial E$ the (signed) distance from the boundary converges to the (signed) distance from the boundary of a Euclidean space after the blow-up we rely on the quasiminimality \cite[Theorem 3.24]{AntonelliPasqualettoPozzetta21} and on \cite[Theorem 2.43]{MoS21}.
\medskip

Finally, in order to prove that $\boldsymbol{\Delta} \dist_E^s\res\partial E=0$ one argues precisely as in the last part of the proof of \cite[Theorem 5.5]{MoS21}.
\end{proof}

\begin{proposition}
Let $(X,\dist,\haus^N)$ be an $\RCD(K,N)$ space for some $K\in\setR$ and $N\ge 2$.
Let $E\subset X$ be a volume constrained minimizer for compact variations in $X$. Let 
\begin{equation}
    \begin{split}
        \mu_\varepsilon^+&:=\varepsilon^{-1}\haus^N\res\{0\leq\dist_{\overline E}<\varepsilon\},\\
        \mu_\varepsilon^-&:=\varepsilon^{-1}\haus^N\res\{0\leq\dist_{E^c}<\varepsilon\}.
    \end{split}
\end{equation}
Then $\mu_\varepsilon^+\to \Per_E$, and $\mu_\varepsilon^-\to\Per_E$ weakly as $\varepsilon\to 0$. In particular, the Minkowski content of $E$ coincides with $\Per(E)$.
\end{proposition}

\begin{proof}

The proof is analogous to the one of \cite[Proposition 6.14]{MoS21}.

The fact that the measures $\mu_{\eps}^+$ and $\mu_{\eps}^-$ have uniformly bounded mass follows from \cite[Lemma 2.42]{MoS21} thanks to \cite[Theorem 3.24]{AntonelliPasqualettoPozzetta21}.

The fact that the perimeter is smaller than any weak limit as $\eps_i\downarrow 0$ of the measures $\mu_{\eps_i}^+$ and $\mu_{\eps_i}^-$ is general and does not require any regularity of $\partial E$. 

In order to prove the converse inequality we rely on \autoref{thm:Isoperimetrici}, which plays the role of \cite[Theorem 5.2]{MoS21} in this setting, on \autoref{prop:laplfull}, which plays the role of \cite[Proposition 5.4]{MoS21} in this setting and argue as in the second part of the proof of \cite[Proposition 6.14]{MoS21}. 
\end{proof}

Our next goal is to turn the Laplacian comparison in \autoref{thm:Isoperimetrici} into an estimate for the perimeter of the equidistant sets from the boundary of an isoperimetric set.

\begin{proposition}\label{prop:variationofarea}
Let us consider an $\RCD(K,N)$ metric measure space $(X,\dist,\haus^N)$ for some $K\in\setR$ and $N\ge 2$.
Let $E\subset X$ be a volume constrained minimizer for compact variations in $X$, and let $c\in\setR$ be given by \autoref{thm:Isoperimetrici}. Then for any $t\ge 0$ it holds
\begin{equation}\label{eq:extareabd}
\Per(\{x\in X\, :\, \dist(x,\overline{E})\le t\})\le J_{c,K,N}(t) \Per(E)\, ,
\end{equation}
and, for any $t\ge 0$,
\begin{equation}\label{eq:intareabound}
\Per(\{x\in X\, :\, \dist(x,X\setminus E)\le t\})\le J_{-c,K,N}(t)\Per(E)\,,
\end{equation}
where we recall that the Jacobian function has been introduced in \eqref{eq:Jacobian}.
\end{proposition}

\begin{proof}
Let us focus on the estimate for the perimeter of the exterior equidistant set, the estimate for the interior one can be obtained with a completely analogous argument.

The bound can be obtained by applying the Gauss--Green integration by parts formula \cite[Theorem 1.6]{BPSGaussGreen} with vector field the gradient of the distance function in the slab $E^t\setminus E$ (compare with \cite[Theorem 5.2]{BPSGaussGreen}, and \autoref{thm:GaussGreen}), where we denoted $E^t:=\{x\in X\, :\, \dist(x,\overline{E})\le t\}$ the $t$-enlargement of $E$. Indeed, taking into account \autoref{thm:Isoperimetrici} we obtain
\begin{equation}\label{eq:HKprimaGronwall}
    \Per(E_t)\le \Per(E)+\int_{E^t\setminus E} (N-1)\frac{s'_{\frac{K}{N-1},-\frac{c}{N-1}}\circ \dist_{\overline{E}}}{s_{\frac{K}{N-1},-\frac{c}{N-1}}\circ \dist_{\overline{E}}}\di \haus^N\, .
\end{equation}
Arguing as in the proof of \cite[Proposition 6.15]{MoS21} via a classical comparison argument for ODEs (see also \cite[Lemma 4.10]{Burtscheretal20}) we obtain that
\begin{equation}
   \Per(\{x\in X\, :\, \dist(x,\overline{E})\le t\})\le J_{c,K,N}(t) \Per(E)\, ,\quad\text{for any $t\ge 0$}\, , 
\end{equation}
as we claimed.
\end{proof}

\begin{remark}
If $(M^n,g)$ is a smooth Riemannian manifold and $E\subset M$ is an isoperimetric set, then 
\begin{equation}
    \lim_{t\to 0^+}\frac{\Per(\{x\in X\, :\, \dist(x,\overline{E})\le t\})-\Per(E)}{t}=c\Per(E)\, 
\end{equation}
and an analogous conclusion holds for the perimeters of the interior equidistant sets. This follows directly from the first variation formula for the perimeter when $\partial E$ is smooth (since we can consider deformations induced by a smooth compactly supported extension of the unit normal of $\partial E$). If $n\ge 8$ an additional approximation argument (relying on the regularity theory for isoperimetric sets) is required, see for instance \cite{Bayle03,NiWangiso}.\\
The validity of an analogous statement for isoperimetric sets in $\RCD(K,N)$ metric measure spaces goes beyond the scope of this note and is left to the future investigation.
\end{remark}

Let us point out the expression for the bounds above when $K=0$. Under these assumptions, with the very same notation above we obtain
\begin{equation}\label{eq:boundareaK0}
\Per(E_t)\le \Per(E)\left(1+\frac{ct}{N-1}\right)^{N-1}\,,\quad\text{for any $t\ge 0$}\, .
\end{equation}

\medskip

Using the coarea formula we can get volume bounds for the tubular neighbourhoods of isoperimetric sets integrating the perimeter bounds in \autoref{prop:variationofarea}.

\begin{corollary}\label{cor:volumebounds}
Let us consider an $\RCD(K,N)$ metric measure space $(X,\dist,\haus^N)$ for some $K\in\setR$ and $N\ge 2$.
Let $E\subset X$ be a volume constrained minimizer for compact variations in $X$, and let $c\in\setR$ be given by \autoref{thm:Isoperimetrici}. Then for any $t\ge 0$ it holds
\begin{equation}\label{eq:extvolbound}
\haus^N(\{x\in X\setminus E\, :\, \dist(x,\overline{E})\le t\})\le \Per(E)\int _0^tJ_{c,K,N}(r)\di r\, ,
\end{equation}
and, for any $t\ge 0$,
\begin{equation}\label{eq:intvolbound}
\haus^N(\{x\in E\, :\, \dist(x,X\setminus E)\le t\})\le \Per(E)\int _0^tJ_{-c,K,N}(r)\di r\, .
\end{equation}

\end{corollary}

\section{Concavity properties of the isoperimetric profile function and consequences}

In \autoref{thm:Isoperimetrici} we proved sharp bounds on the Laplacian of the distance function from an isoperimetric set. Such bounds encode information about the first and second variation of the area of equidistants from the isoperimetric boundary. As we shall see, this information is sufficient to extend the sharp concavity properties for the isoperimetric profile known for smooth and compact Riemannian manifolds with lower Ricci curvature bounds to the setting of $N$-dimensional compact $\RCD(K,N)$ spaces.\\
More in general, we are going to join such information together with the generalized asymptotic mass decomposition \autoref{thm:MassDecompositionINTRO} to get sharp concavity properties of the isoperimetric profile for $\RCD(K,N)$ spaces $(X,\dist,\haus^N)$ with a lower bound on the volume of unit balls in \autoref{thm:BavardPansu}, which is the main result of this section. In particular, this will imply that the sharp concavity properties of the isoperimetric profile hold on complete non-compact manifolds with uniform lower Ricci curvature bounds and uniform lower volume bounds.

Concavity properties of the isoperimetric profile for (weighted) manifolds with lower Ricci curvature bounds have been considered by various authors, see for instance \cite{BavardPansu86,Bayle03,Bayle04,BayleRosales,Milmanconvexity,NiWangiso,MorganJohnson00,Gallotast}. All these works deal with compact manifolds or with weighted manifolds of finite total measure and they heavily rely on the existence of isoperimetric regions for any volume. In all cases smoothness is a relevant assumption, in order to rely on the regularity theory for isoperimetric regions. The case of non-smooth weights in \cite{Milmanconvexity} is handled with a careful approximation procedure. The only previous references where the problem is considered in non-compact manifolds with infinite volume are \cite{RitoreExistenceSurfaces01,Nar14} and \cite{LeonardiRitore}. In \cite{RitoreExistenceSurfaces01} the case of surfaces with non-negative Gaussian curvature is treated and existence of isoperimetric regions of any volume is an intermediate step in order to prove concavity of the isoperimetric profile. In \cite{Nar14} the case of complete non-compact manifolds with $C^{2,\alpha}$ bounded geometry is considered. In \cite{LeonardiRitore} the authors consider unbounded convex bodies $C\subset \setR^n$ verifying no further regularity assumptions. Their proof relies on a generalized existence result for isoperimetric regions and on an approximation argument, to deal with convex bodies with non-smooth boundary.
\medskip

The two statements below are proved in \cite{AntonelliNardulliPozzetta} building on top of \cite{Nar14, AFP21, AntonelliPasqualettoPozzetta21}. They will be key ingredients for the proof of \autoref{thm:BavardPansu}.

\begin{theorem}[Asymptotic mass decomposition]\label{thm:MassDecompositionINTRO}
Let $(X,\dist,\haus^N)$ be a non-compact $\RCD(K,N)$ space. Assume there exists $v_0>0$ such that $\haus^N(B_1(x))\geq v_0$ for every $x\in X$. Let $V>0$. For every minimizing (for the perimeter) sequence $\Omega_i\subset X$ of volume $V$, with $\Omega_i$ bounded for any $i$, up to passing to a subsequence, there exist an increasing and bounded sequence $\{N_i\}_{i\in\mathbb N}\subseteq \mathbb N$, disjoint finite perimeter sets $\Omega_i^c, \Omega_{i,j}^d \subset \Omega_i$, and points $p_{i,j}$, with $1\leq j\leq N_i$ for any $i$, such that
\begin{itemize}
    \item $\lim_{i} \dist(p_{i,j},p_{i,\ell}) = \lim_{i} \dist(p_{i,j},o)=\infty$, for any $j\neq \ell<\overline N+1$ and any $o\in X$, where $\overline N:=\lim_i N_i <\infty$;
    \item $\Omega_i^c$ converges to $\Omega\subset X$ in the sense of finite perimeter sets, and we have $\haus^N(\Omega_i^c)\to_i \haus^N(\Omega)$, and $ \Per( \Omega_i^c) \to_i \Per(\Omega)$. Moreover $\Omega$ is a bounded isoperimetric region for its own volume in $X$;
    \item for every $j<\overline N+1$, $(X,\dist,\haus^N,p_{i,j})$ converges in the pmGH sense  to a pointed $\RCD(K,N)$ space $(X_j,\dist_j,\haus^N,p_j)$. Moreover there are isoperimetric regions $Z_j \subset X_j$ such that $\Omega^d_{i,j}\to_i Z_j$ in $L^1$-strong and $\Per(\Omega^d_{i,j}) \to_i \Per (Z_j)$;
    \item it holds that
    \begin{equation}\label{eq:UguaglianzeIntro}
    I_{(X,\dist,\haus^N)}(V) = \Per(\Omega) + \sum_{j=1}^{\overline{N}} \Per (Z_j),
    \qquad\qquad
    V=\haus^N(\Omega) +  \sum_{j=1}^{\overline{N}} \haus^N(Z_j).
    \end{equation}
\end{itemize}
\end{theorem}

\begin{proposition}\label{prop:ProfileDecomposition}
Let $(X,\dist,\haus^N)$ be a non-compact $\RCD(K,N)$ space. Assume there exists $v_0>0$ such that $\haus^N(B_1(x))\geq v_0$ for every $x\in X$. Let $\{p_{i,j} \st i \in \mathbb N\}$ be a sequence of points on $X$, for $j=1,\ldots,\overline{N}$ where $\overline{N}\in\N \cup \{+\infty\}$. Suppose that each sequence $\{p_{i,j}\}_i$ is diverging along $X$ and that $(X,\dist,\haus^N,p_{i,j})$ converges in the pmGH sense  to a pointed $\RCD(K,N)$ space $(X_j,\dist_j,\haus^N,p_j)$. Defining
\begin{equation}\label{eqn:GeneralizedIsoperimetricProfile}
I_{X\sqcup_{j=1}^{\overline N}X_j}(v):=\inf\left\{\Per(E)+\sum_{j=1}^{\overline N}\Per (E_j):E\subseteq X,E_j\subseteq X_j, \haus^N(E)+\sum_{j=1}^{\overline N}\haus^N(E_j)=v\right\},
\end{equation}
it holds $I_{X\sqcup_{j=1}^{\overline N} X_j}(v) = I_X(v)$ for any $v>0$.
\end{proposition}

We need to start with a mild regularity property of the isoperimetric profile, namely that it is strictly positive and continuous. Later in the paper these two properties will be sharpened in several directions. We stress that an argument similar to the one discussed in \autoref{lem:LocalHolderProfile} had already appeared in \cite[Lemma 6.9]{Milmanconvexity}, and \cite[Lemma 6.2]{Gallotast}. By a careful inspection of the proofs, the argument for proving the local H\"older property in \autoref{lem:LocalHolderProfile} is likely to be adapted in the more general case of $\CD(K,N)$ spaces $(X,\dist,\meas)$ with densities uniformly bounded above and volume of unit balls uniformly bounded below, as kindly pointed out to the authors by E. Milman. Since we do not need such level of generality, we will not give the details of the proof in such a general case.

\begin{lemma}\label{lem:LocalHolderProfile}
Let $(X,\dist,\haus^N)$ be an $\RCD(K,N)$ space. Assume that there exists $v_0>0$ such that $\haus^N(B_1(x))\geq v_0$ for every $x\in X$. Let $I_X:(0,\haus^N(X))\to \mathbb R$ be the isoperimetric profile of $X$. Then $I_X(v)>0$ for every $v>0$ and $I_X$ is continuous.
\end{lemma}

\begin{proof}
The first conclusion can be reached arguing as in \cite[Remark 4.7]{AFP21}, building on the the top of \autoref{thm:MassDecompositionINTRO}.
The second conclusion can be reached adapting \cite[Theorem 2]{FloresNardulli20}, which shows that $I$ is locally $(1-\frac{1}{N})$-H\"older, cf. \cite[Lemma 2.23]{AntonelliNardulliPozzetta}.
\end{proof}

\subsection{Sharp concavity inequalities for the isoperimetric profile}

Given a continuous function $f:(0,\infty)\to (0,\infty)$ and parameters $K\in\setR$ and $1< N<\infty$ we are going to consider second order differential inequalities of the form
\begin{equation}\label{eq:BPf}
-f''f\ge K+\frac{(f')^2}{N-1}    
\end{equation}
and 
\begin{equation}\label{eq:Bf}
-f''\ge \frac{KN}{N-1}f^{\frac{2-N}{N}}\, .
\end{equation}
In general the function $f$ will be continuous but not twice differentiable everywhere and the inequalities will be understood in the viscosity sense, i.e. we will require that whenever $\phi:(x_0-\eps,x_0+\eps)\to \setR$ is a $C^2$ function with $\phi\le f$ on $(x_0-\eps,x_0+\eps)$ and $\phi(x_0)=f(x_0)$ the corresponding inequality \eqref{eq:BPf} or \eqref{eq:Bf} holds at $x_0$ with $\phi$ in place of $f$.

\begin{theorem}\label{thm:BavardPansu}
Let $(X,\dist,\haus^N)$ be an $\RCD(K,N)$ space. Assume that there exists $v_0>0$ such that $\haus^N(B_1(x))\geq v_0$ for every $x\in X$. 

Let $I:(0,\haus^N(X))\to (0,\infty)$ be the isoperimetric profile of $X$. Then 
\begin{enumerate}
    \item The inequality
    \begin{equation}\label{eqn:BP}
    -I''I\geq K+\frac{(I')^2}{N-1}\,\quad\text{holds in the viscosity sense on $(0,\haus^N(X))$}\, ;
    \end{equation}
    \item Let $N\leq \alpha <+\infty$ and $\xi:=I^{\frac{\alpha}{\alpha-1}}$. Hence the inequality 
    \begin{equation}\label{eqn:BayleEmpowered}
    \begin{split}
           -\xi''&\geq \frac{\alpha}{\alpha-1}\xi^{\frac{2-\alpha}{\alpha}}\left(\left(\frac{1}{N-1}-\frac{1}{\alpha-1} \right)(I')^2 + K \right) \\
           &= \frac{\alpha}{\alpha-1}\xi^{\frac{2-\alpha}{\alpha}}\left(\left(\frac{1}{N-1}-\frac{1}{\alpha-1} \right)\frac{(\alpha-1)^2}{\alpha^2} \xi^{-\frac{2}{\alpha}}(\xi')^2 + K \right)
           ,
    \end{split}
    \end{equation}
    holds in the viscosity sense on $(0,\haus^N(X))$.\\
    In particular, in the case $\alpha=N$, if $\psi:=I^{\frac{N}{N-1}}$ then 
    \begin{equation}\label{eqn:Bayle}
    -\psi''\geq \frac{KN}{N-1}\psi^{\frac{2-N}{N}}\,\quad\text{holds in the viscosity sense on $(0,\haus^N(X))$}\, .
    \end{equation}
\end{enumerate}
\end{theorem}

\begin{proof}
Let us assume $\haus^N(X)=+\infty$, the compact case being completely analogous.
Let us prove \eqref{eqn:BP} first.
Let $v\in (0,\infty)$ be fixed. Take $\Omega_i$ a minimizing sequence of bounded sets of volume $v$ and let $\Omega,Z_j$ be the isoperimetric regions in $X,X_j$ respectively, according to the notation of \autoref{thm:MassDecompositionINTRO}.

Let us consider a smooth function $\phi$ such that $\phi\le I$ in a neighbourhood of $v$ and $\phi(v)=I(v)$. We wish to prove that
\begin{equation}\label{eq:diffineq}
-\phi''(v)\phi(v)\ge K+ \frac{\left(\phi'(v)\right)^2}{N-1}\, .
\end{equation}
Let $E$ be one of the isoperimetric sets $\Omega,Z_j$, and let $E_t$ be the $t$-enlargements, in the associated space $X,X_j$, of $E$ for $t\in (-\eps,\eps)$. To be more precise, for $t\geq 0$, the points in $E_t$ are those points of $X$ that have distance $\leq t$ from $\overline E$, while, for $t\leq 0$, the points in $E_t$ are those points of $X$ such that the distance from $X\setminus E$ is $\geq -t$.
Let $c$ be a mean curvature barrier for $E$ provided by \autoref{thm:Isoperimetrici}. We now show that $c=\phi'(v)$.\\
As we observed in \eqref{eq:extareabd},
\begin{equation}\label{eqn:PerEt}
\Per(E_t)\le J_{c,K,N}(t)\Per(E) ,\quad\text{for any $t\in (-\eps,\eps)$}\, .
\end{equation}
By simple computations, 
\begin{equation}\label{eq:use}
J'_{c,K,N}(0)=c\, \quad \text{and}\quad  J''_{c,K,N}(0)=-K+\frac{N-2}{N-1}c^2\, .
\end{equation}
Let us show that $t\mapsto \Per(E_t)$ is differentiable at $t=0$. 

Setting $\beta(t):=\haus^N(E_t)+\sum_{T\in\{\Omega,Z_1,\dots,Z_{\overline N}\}, T\neq E}\haus^N(T)$, by the coarea formula we get that $\beta(t)$ is continuous at $t=0$. Hence, by \eqref{eqn:GeneralizedIsoperimetricProfile}, and \eqref{eqn:PerEt}, we have that, for any $t\in (-\eps,\eps)$,
\begin{equation}\label{eqn:FINAL}
\begin{split}
J_{c,K,N}(t)\Per(E)&+\sum_{T\in\{\Omega,Z_1,\dots,Z_{\overline N}\}, T\neq E}\Per(T)\geq 
\Per(E_t)+\sum_{T\in\{\Omega,Z_1,\dots,Z_{\overline N}\}, T\neq E}\Per(T) \\ &\ge I_{X\sqcup_{j=1}^{\overline N}X_j}(\beta(t))=I(\beta(t))\ge \phi(\beta(t))\, .
\end{split}
\end{equation}
From \eqref{eqn:FINAL}, the continuity of $\beta(t)$ at $t=0$, the fact that \begin{equation}
\phi(v)=I(v)=\Per(E)+\sum_{T\in\{\Omega,Z_1,\dots,Z_{\overline N}\}, T\neq E}\Per(T),
\end{equation}
we first deduce that $\Per(E_t)$ is continuous at $t=0$. Hence, by the continuity of $t\mapsto \Per(E_t)$ at $t=0$ and the coarea formula, $t\mapsto \beta(t)$ is differentiable at $t=0$. Exploiting this information again in \eqref{eqn:FINAL} we obtain that $t\mapsto \Per(E_t)$ is differentiable at $t=0$.

By using the latter differentiability together with \eqref{eqn:PerEt} we get that 
\begin{equation}
\frac{\di }{\di t}\Per(E_t)|_{t=0}=c\Per(E)\, .
\end{equation}
By coarea, this implies that
\begin{equation}
t\mapsto\haus^N(E_t) \, ,
\end{equation}
is twice differentiable at $t=0$ with 
\begin{equation}
\frac{\di }{\di t}\haus^N(E_t)|_{t=0}=\Per(E)\, \quad\text{and}\quad
\frac{\di ^2}{\di t^2}\haus^N(E_t)|_{t=0}=c\Per(E)\, .
\end{equation}

Moreover, from \eqref{eq:UguaglianzeIntro} and the discussion above, $\beta(0)=v$, $\beta'(0)=\Per(E)$, $\beta''(0)=c\Per(E)$ and $\phi'(v)\beta'(0)=c\Per(E)$, therefore $\phi'(v)=c$, which is the sought claim.

Hence, we showed that the barrier $c$ given by \autoref{thm:Isoperimetrici} is unique and $c:=\varphi'(v)$ for every isoperimetric set $\Omega,Z_j$ in any of the limit spaces $X,X_j$. 
\medskip

Let us define $\widetilde\beta(t):=\haus^N(\Omega_t)+\sum_{j=1}^{\overline N}\haus^N((Z_j)_t)$ for every $t\in (-\varepsilon,\varepsilon)$ with $\varepsilon>0$ small enough.

Arguing as in \eqref{eqn:FINAL} for each isoperimetric region separately, we get that $\widetilde\beta(t)$ is twice differentiable at $t=0$. Moreover
\begin{equation}\label{eqn:LEI}
\phi(\widetilde\beta(t))\le J_{c,K,N}(t)\left(\Per(\Omega)+\sum_{j=1}^{\overline N}\Per(Z_j)\right)\, \quad\text{for any $t\in (-\eps,\eps)$}\, ,
\end{equation}
$\widetilde\beta(0)=v$, $\widetilde\beta'(0)=I(v)$, and $\widetilde\beta''(0)=cI(v)$.
Using \eqref{eqn:LEI}, we can pass to the limit as $t\to0$ in
\begin{equation}
\begin{split}
    \frac{\phi(\widetilde\beta(t))+\phi(\widetilde\beta(-t))-2\phi(\widetilde\beta(0))}{t^2} &=   \frac{\phi(\widetilde\beta(t))+\phi(\widetilde\beta(-t))-2I(v)}{t^2} \\& \le \frac{J_{c,K,N}(t)+J_{c,K,N}(-t)-2}{t^2}I(v)\, ,
\end{split}
\end{equation}
to obtain
\begin{equation}
\phi''(\widetilde\beta(0))\widetilde\beta'(0)^2+\phi'(\widetilde\beta(0))\widetilde\beta''(0)\le I(v)\left(c^2\frac{N-2}{N-1}-K\right)\, ,
\end{equation}
Using the previous identities, we obtain the sought conclusion 
\begin{equation}
    \phi''(v)\phi(v)=\phi''(v)I(v)\leq -\frac{\phi'(v)^2}{N-1}-K\, .
\end{equation}
\medskip

Let us prove \eqref{eqn:BayleEmpowered}. We need to prove that for every $v\in (0,\haus^N(X))$, if $\varphi$ is a smooth function in a neighbourhood $U$ of $v$ such that $\phi\leq\xi$ on $U$ and $\phi(v)=\xi(v)$, then \eqref{eqn:BayleEmpowered} holds with $\varphi$ in place of $\xi$. The latter inequality is obtained noticing that $\phi^{\frac{\alpha-1}{\alpha}}\leq I$ on $U$ and $\phi^{\frac{\alpha-1}{\alpha}}(v)=I(v)$, using \eqref{eqn:BP} with some easy algebraic computations.
\end{proof}

\begin{remark}
Let us point out that \eqref{eqn:BayleEmpowered}, and thus \eqref{eqn:Bayle}, holds also in the sense of second order incremental quotients considered in \cite{Bayle03,Bayle04}. Given a continuous function $f:(0,\haus^N(X))\to(0,\infty)$ and $x\in (0,\haus^N(X))$ we denote
\begin{equation}
\overline{D^2}f(x):=\limsup_{h\downarrow 0}\frac{f(x+h)+f(x-h)-2f(x)}{h^2}\, .
\end{equation}
Then, for example, the very same proof of \eqref{eqn:Bayle}, where the bound was considered in the viscosity sense, shows that 
\begin{equation}
\overline{D^2}\psi\le -\frac{KN}{N-1}\psi^{\frac{2-N}{N}}\, \quad\text{on $(0,\haus^N(X))$}\, .
\end{equation}
\end{remark}

\subsection{Fine properties of the isoperimetric profile}

In this subsection we derive further regularity properties and asymptotics for small volumes of the isoperimetric profile. They will be particularly useful to study the stability of isoperimetric regions under non-collapsed (pointed) Gromov--Hausdorff convergence,  diameter bounds and connectedness properties for isoperimetric regions, and the asymptotic isoperimetric behaviour of non-collapsed spaces with lower Ricci curvature bounds (see also the forthcoming \cite{AntonelliPasqualettoPozzettaSemolaASYMPTOTIC}).\\
The arguments essentially rely only on the sharp concavity properties of the isoperimetric profile and on qualitative isoperimetric inequalities.
\medskip

The following corollary is a standard consequence of \autoref{thm:BavardPansu}, and therefore we omit its proof.

\begin{corollary}\label{cor:mildpropprofile}
Let $(X,\dist,\haus^N)$ be an $\RCD(K,N)$ space with $K\leq 0$. Assume that there exists $v_0>0$ such that $\haus^N(B_1(x))\geq v_0$ for every $x\in X$.

For every $v\in(0,\haus^N(X))$ and for every $\delta\in(v,\haus^N(X)-v)$ there exists $C>0$ such that the function $I^{\frac{N}{N-1}}(x)-Cx^2$ is concave on $(v-\delta,v+\delta)$. Hence the isoperimetric profile function $I$ has right derivative $I_+'(v)$ and left derivative $I_-'(v)$ defined for every $v\in (0,\haus^N(X))$. Moreover the isoperimetric profile $I$ is differentiable in all $(0,\haus^N(X))$ except at most countably many values, it is locally Lipschitz, and it is twice differentiable almost-everywhere. Moreover \eqref{eqn:BP}, \eqref{eqn:BayleEmpowered}, and \eqref{eqn:Bayle} also hold pointwise almost everywhere.
\end{corollary}

{ We now aim at giving a slight improvement of \autoref{cor:mildpropprofile}. We first need an auxiliary result.\\
The following \autoref{lem:ConsequenceSmallDiameter}
is rather classical and it holds in the class of locally doubling metric measure spaces satisfying a Poincaré inequality and a uniform noncollapsing assumption on the volumes of unit balls. Such a result has its roots in the papers \cite{Buser82, Kanai86, ChavelFeldman91}, and \cite{CoulhonSaloffCoste95}. Since we only need it in the setting of $N$-dimensional $\RCD(K,N)$ spaces, we state it in this setting. See also \cite[Proposition 3.20]{AntonelliPasqualettoPozzetta21} (cf. \cite[Remark 3.21]{AntonelliPasqualettoPozzetta21}).

%

\begin{lemma}\label{lem:ConsequenceSmallDiameter}
Let $K\in\mathbb R$ and $N\geq 2$. Let $(X,\dist,\haus^N)$ be an $\RCD(K,N)$ space.
Let us assume there exists $v_0>0$ such that $\haus^N(B_1(x))\geq v_0$ for every $x\in X$. Then there exist $v_1:=v_1(K,N,v_0)$, and $\vartheta:=\vartheta(K,N,v_0)$ such that
\begin{equation}\label{eqn:BoundBelowIsopProfile}
I(v)\geq \vartheta v^{\frac{N-1}{N}},\qquad \text{for all $v\leq v_1$}.
\end{equation}
\end{lemma}

\begin{remark}\label{rmk:ConseqMJ}
Under the assumptions of \autoref{lem:ConsequenceSmallDiameter}, by bounding from above $I$ with the perimeter of balls, we have that there exist constants \(C_0=C_0(K,N,v_0)>0\) and \(\bar v=\bar v(K,N,v_0)>0\) such that
\begin{equation}\label{eq:EstRatio}
\frac{I(v)}{v^{\frac{N-1}{N}}}\leq C_0,\quad\text{ for every }v\leq\bar v.
\end{equation}
%
%
\end{remark}
}

Joining the second order differential inequalities derived in \autoref{thm:BavardPansu} with \autoref{lem:ConsequenceSmallDiameter} we derive further analytical properties of the isoperimetric profile.

\begin{proposition}\label{prop:SharpenedConcavityAndLimit}
 Let $K\leq 0$ and $N\geq 2$. Let $(X,\dist,\haus^N)$ be an $\RCD(K,N)$ space such that there exists $v_0\geq 0$ with $\haus^N(B_1(x))\geq v_0$ for every $x\in X$. Then the following hold.
 \begin{enumerate}
 \item There exist $C:=C(K,N,v_0)>0$ and $v_1:=v_1(K,N,v_0)>0$ such that the function $\eta(v):=I^{\frac{N}{N-1}}(v)-Cv^{\frac{2+N}{N}}$ is concave on the interval $[0,v_1]$. Moreover, if $N=2$, we can choose $C=-K$ and the claim holds on $[0,\haus^N(X)]$, if $\haus^N(X)<\infty$, or on $[0,\haus^N(X))$, if $\haus^N(X)=\infty$.\\
 As a consequence the function
 \[
 [0,\haus^N(X))\ni v\mapsto \frac{I(v)}{v^{\frac{N-1}{N}}},
 \]
 has a finite strictly positive limit as $v\to 0$.


{
\item There exists $\tilde v_1 := \tilde v_1(K,N,v_0) \in (0,v_1]$ such that $I$ is concave on $[0,\tilde v_1]$.}
 \end{enumerate}
 \end{proposition}
 
 \begin{proof}
 Since the proof follows from \autoref{thm:BavardPansu}, \autoref{lem:ConsequenceSmallDiameter} and an elementary one-dimensional analysis, we just sketch it.\\
 Let us prove item (1). Let us first deal with the case $N>2$. It suffices to take 
 \begin{equation}\label{eqn:ChoiceConst}
 C:=\frac{-KN^3\vartheta^{\frac{2-N}{N-1}}}{2(N-1)(N+2)}, \, \text{where $\vartheta$ and $v_1$ are the constants in \eqref{eqn:BoundBelowIsopProfile}}\, .
 \end{equation}
 With this choice of the constant $C$, it can be readily proved that $-\eta''\geq 0$ holds in the viscosity sense on $(0,v_1)$. This is proved by a straightforward computation using item (2) of \autoref{thm:BavardPansu}, \eqref{eqn:BoundBelowIsopProfile}, and the choice of the constant $C$. The conclusion then follows since $\eta$ is continuous on $(0,v_1)$ thanks to \autoref{lem:LocalHolderProfile}.
 %
%
%
  

Let us deal with the remaining case $N=2$. From \autoref{thm:BavardPansu} we get that the function $\psi:=I^{2}$ satisfies 
 \[
 -\psi''\geq 2K,
 \]
 in the viscosity sense on $(0,\haus^N(X))$. 
 Hence the function $\eta(x):=\psi(x)+Kx^2$ satisfies $-\eta''\geq 0$ in the viscosity sense on $(0,\haus^N(X))$. Therefore, in this case we can take $C:=-K$ on the whole interval $[0,\haus^N(X)]$ or $[0,\haus^N(X))$, depending on whether $\haus^N(X)$ is finite or infinite.

 The last conclusion in the statement of item (1) readily follows from the fact that $v\mapsto\eta(v)/v$ is non-increasing on $[0,v_1]$, since $\eta$ is concave on $[0,v_1]$.
%

 \medskip

{
Let us now prove item (2). By concavity of $\eta$, exploiting \eqref{eqn:BoundBelowIsopProfile} and \eqref{eq:EstRatio}, for $A>1$ we find
\begin{equation}
    \begin{split}
        \eta'_+(v) & \ge \frac{\eta(Av) - \eta(v)}{Av - v} 
        \ge \frac{1}{(A-1)v} \bigg(
        \vartheta^{\frac{N}{N-1}} Av - C A^{\frac{2+N}{N}} v^{\frac{2+N}{N}} - C_0^{\frac{N}{N-1}}v + C v^{\frac{2+N}{N}}
        \bigg)\\
        &\ge \vartheta^{\frac{N}{N-1}}  - \frac{C_0^{\frac{N}{N-1}}}{A-1} + \frac{1}{(A-1)v} \bigg(
        - C A^{\frac{2+N}{N}} v^{\frac{2+N}{N}}  + C v^{\frac{2+N}{N}}
        \bigg),
    \end{split}
\end{equation}
for any $0<v<Av< v_1$. Hence choosing first $A$ sufficiently large and then restricting $v \in (0,\tilde v_1]$ for $\tilde v_1 < v_1$ small enough, we obtain that $\eta'_+(v) \ge \tfrac12\vartheta^{\frac{N}{N-1}}$ on $v \in (0,\tilde v_1]$. This implies $\tfrac{N}{N-1} I^{\frac{1}{N-1}}(v) I'_+(v) \ge \tfrac12 \vartheta^{\frac{N}{N-1}}$ on $v \in (0,\tilde v_1]$, and thus
\begin{equation}
I'_+(v) \ge \frac{N-1}{2N} \vartheta^{\frac{N}{N-1}} \Big(C_0^{\frac{1}{N-1}} v^{\frac1N}\Big)^{-1}.
\end{equation}
Therefore \eqref{eqn:BP} implies that $I''\le0$ on $(0,\tilde v_1]$ in the viscosity sense, up to decrease $\tilde v_1$, proving item (2).
}
\end{proof}

\begin{remark}
Item (1) of \autoref{prop:SharpenedConcavityAndLimit} answers in the affirmative to Questions 2 and 3 in \cite{NardulliOsorio20} in the more general setting of $N$-dimensional $\RCD(K,N)$ spaces. As a consequence, it is possible to drop the additional hypothesis (H) in \cite[Lemma 4.9]{NardulliOsorio20}.
\end{remark}

Let us now derive some further consequences from the concavity properties above.

\begin{proposition}\label{cor:EstimateDerivativeProfileAndSubadditivity}
Let $K\leq 0$ and $N\geq 2$. Let $(X,\dist,\haus^N)$ be an $\RCD(K,N)$ space such that there exists $v_0\geq 0$ such that $\haus^N(B_1(x))\geq v_0$ for every $x\in X$. Let us denote $\vartheta:=\lim_{v\to 0}I(v)/v^{\frac{N-1}{N}}>0$, which exists due to item (1) of \autoref{prop:SharpenedConcavityAndLimit}. Hence the following hold.
\begin{enumerate}
    \item 
    There holds
    \[
    \lim_{v\to0} \frac{I'_+(v)}{v^{-\frac1N}} = \frac{N-1}{N} \vartheta.
    \]
    \item Let $\alpha>N$. Hence there exists
    $\varepsilon=\eps(K,N,v_0,\alpha)>0$
    such that $I^{\frac{\alpha}{\alpha-1}}$ is concave on $(0,\varepsilon)$. As a consequence, $I$ is strictly subadditive on $(0,\varepsilon)$.
\end{enumerate}
\end{proposition}

\begin{proof}
{The first item follows from elementary computations { exploiting that $I$ is concave for small volumes by \autoref{prop:SharpenedConcavityAndLimit}}.
The second claim now easily follows from the first one, employing \eqref{eqn:BayleEmpowered}, {\autoref{lem:ConsequenceSmallDiameter} and \autoref{rmk:ConseqMJ}}. 
}
\end{proof}

\begin{remark}\label{rem:AsintoticaProfileSmallVolumes}
    Under the same assumptions of \autoref{cor:EstimateDerivativeProfileAndSubadditivity}, we prove in \cite{AntonelliPasqualettoPozzettaSemolaASYMPTOTIC} that the isoperimetric profile satisfies the asymptotic behavior for small volume given by
    \begin{equation}\label{eq:AsintoticaProfileSmallVolumes}
    \lim_{v\to 0}\frac{I_X(v)}{v^{\frac{N-1}{N}}} = N\left(\omega_N\vartheta_{\infty,\mathrm{min}}\right)^{\frac 1N}\, ,
    \end{equation}
    where, being $v(N,K/(N-1),r)$ the volume of the ball of radius $r$ in the simply connected model space with constant sectional curvature $K/(N-1)$ and dimension $N$, we have that 
    \begin{equation*}    
    \vartheta_{\infty,\mathrm{min}}\eqdef \lim_{r\to 0}\inf_{x\in X}\frac{\haus^N(B_r(x))}{v(N,K/(N-1),r)}\, >\, 0\, 
    \end{equation*}    
    is the minimum of all the possible densities at any point in $X$ or in any pmGH limit at infinity of $X$. The limit in \eqref{eq:AsintoticaProfileSmallVolumes} yields an answer to Questions 4 in \cite{NardulliOsorio20}.\\
    It follows that the limit $\vartheta$ in \autoref{cor:EstimateDerivativeProfileAndSubadditivity} is now known to be equal to $N(\omega_N \vartheta_{\infty,\min})^{\frac1N}$. Hence
    \begin{equation*}
    \lim_{v\to0} \frac{I'_+(v)}{v^{-\frac1N}} = (N-1) (\omega_N \vartheta_{\infty,\min})^{\frac1N}\, .
    \end{equation*}
\end{remark}


{ The following lower bound appears to be classical, and it could be stated and proved in the class of locally doubling metric measure spaces satisfying a Poincaré inequality, and a uniform noncollapsing assumption on the volumes of unit balls. We refer to \cite[Theorem V.2.6]{ChavelIsoperimetricBook01}, which is stated for smooth Riemannian manifolds with bounded geometry, but whose proof adapts to the latter setting. Since we only need the statement in the setting of $N$-dimensional $\RCD(K,N)$ spaces, we state it in this setting. We stress that an alternative proof of \autoref{cor:UniformPositivityProfile} using the results of this paper can be given arguing by contradiction, by exploiting \autoref{thm:MassDecompositionINTRO} and \autoref{cor:EstimateDerivativeProfileAndSubadditivity}.
%
%
\begin{corollary}\label{cor:UniformPositivityProfile}
Let $0<V_1<V_2<V_3$ and let $K \in \R, N \in \N_{\ge 2}, v_0>0$. Then there exists $\mathscr{I}=\mathscr{I}(K,N,v_0,V_1,V_2, V_3)>0$ such that the following holds. If $(X,\dist,\haus^N)$ is an $\RCD(K,N)$ space with $\inf_{x\in X}\haus^N(B_1(x))\geq v_0>0$ and $\haus^N(X)\ge V_3$, then
\begin{equation}
    I_X(v) \ge \mathscr{I} \qquad\forall\,v \in [V_1,V_2]\, .
\end{equation}
\end{corollary}
}

\begin{corollary}\label{cor:UniformLipschitzProfile}
Let $0<V_1<V_2$ and let $K \le 0, N \in \N_{\ge 2}, v_0>0$. Then there exist $\mathscr{C},\mathscr{L}>0$ depending on $K,N,v_0,V_1,V_2$ such that the following holds. If $(X,\dist,\haus^N)$ is $\RCD(K,N)$ with $\inf_{x\in X}\haus^N(B_1(x))\geq v_0>0$ and $\haus^N(X)\ge V_2$, then
\begin{equation}
    \begin{split}
        v\mapsto I^{\frac{N}{N-1}}(v) - \mathscr{C}v^{\frac{2+N}{N}} &\quad\text{is concave on $[0,V_1]$}\, ,\\
       v\mapsto  I^{\frac{N}{N-1}}(v) &\quad\text{is $\mathscr{L}$-Lipschitz on $[0,V_1]$}\, .
    \end{split}
\end{equation}
In particular, for any $V \in (0,V_1)$ there is $L>0$ depending on $K,N,v_0,V_1,V_2, V$ such that
\begin{equation}
      v\mapsto  \text{$I(v)$ is $L$-Lipschitz on $[V,V_1]$}\, .
\end{equation}
\end{corollary}

\begin{proof}
{
We know from \autoref{prop:SharpenedConcavityAndLimit} that there exist $v_1,C>0$ depending on $K,N,v_0$ such that $I^{\frac{N}{N-1}}(v) - Cv^{\frac{2+N}{N}}$ is concave on $[0,v_1]$.\\ 
By \eqref{eqn:Bayle} and \autoref{cor:UniformPositivityProfile}, easy computations give that $I^{\frac{N}{N-1}}(v) - \mathscr{C} v^{\frac{2+N}{N}}$ is concave on $[0,V_1]$, for a constant $\mathscr{C}=\mathscr{C}(K,N,v_0,V_1,V_2)>0$.

Let us denote $f(v)\eqdef I^{\frac{N}{N-1}}(v) - \mathscr{C} v^{\frac{2+N}{N}} $.
In order to show the Lipschitzianity of $I^{\frac{N}{N-1}}$, it is enough to observe that by concavity, 
\autoref{rmk:ConseqMJ},
\autoref{prop:SharpenedConcavityAndLimit},
\autoref{cor:EstimateDerivativeProfileAndSubadditivity}, and \autoref{cor:UniformPositivityProfile}, we have that $f'_+$ is bounded above on $[0,V_1]$, and $f$ is bounded below on $[V_1/2,(V_1+V_2)/2]$. Hence, since $f$ is concave, we get that $f'_+$ is uniformly bounded above and below on $[0,V_1]$ by constants only depending on $K,N,v_0,V_1,V_2$.\\ 
As a direct consequence $I^{\frac{N}{N-1}}$ is $\mathscr{L}$-Lipschitz on $[0,V_1]$ for some constant $\mathscr{L}$ depending on $K,N,v_0,V_1,V_2$.
%

}
\end{proof}

The above uniform bounds on the isoperimetric profile allow to derive the following uniform regularity properties on isoperimetric regions, that is, we prove that isoperimetric sets satisfy almost-minimality properties and density estimates with constants \emph{independent} of the specific ambient space. This is new even in the case of isoperimetric sets in smooth Riemannian manifolds.

\begin{corollary}\label{cor:UniformRegularityIsop}
Let $0<V_1<V_2<V_3$ and let $K \in \R, N \in \N_{\ge 2}, v_0>0$. Then there exist $\Lambda, R>0$ depending on $K,N,v_0,V_1,V_2, V_3$ such that the following holds.

If $(X,\dist,\haus^N)$ is an $\RCD(K,N)$ space with $\inf_{x\in X}\haus^N(B_1(x))\geq v_0>0$, $\haus^N(X)\ge V_3$, and $E\subset X$ is an isoperimetric region with $\haus^N(E) \in [V_1,V_2]$, then $E$ is a $(\Lambda,R)$-minimizer, i.e., for any $F\subset X$ such that $F \Delta E \subset B_R(x)$ for some $x \in X$, then
\[
\Per(E) \le \Per(F) + \Lambda\haus^N(F\Delta E)\, .
\]
In particular there exist $R'>0$, $C_1 \in (0,1)$ and $C_2,C_3>0$ depending on $K,N,v_0,V_1,V_2, V_3$ such that
\begin{equation}\label{eqn:OmegaMinimo}
\Per(E,B_r(x))\leq (1+C_3r)\Per(F,B_r(x)),
\end{equation}
for any $x\in X$, $r\in (0,R']$, and any $F$ such that $E\Delta F \Subset B_r(x)$. Moreover
\[
C_1 \le \frac{\haus^N(B_r(x) \cap E)}{\haus^N(B_r(x))}\le 1-C_1\, ,
\qquad
C_2^{-1} \le \frac{\Per(E,B_r(x))}{r^{N-1}} \le C_2\, ,
\]
for any $x \in \partial E$ and any $r\in(0,R']$.

\end{corollary}

\begin{proof}
Let $R>0$ be a radius such that $\haus^N(B_R(x)) \le \min\{V_1/2,(V_3-V_2)/2 \}$ for any $x \in X$. Let also $L>0$ be the Lipschitz constant of $I$ on $[V_1/2, (V_2+V_3)/2]$ given by \autoref{cor:UniformLipschitzProfile}. Then for any $F\subset X$ with $F\Delta E\subset B_R(x)$ it holds
\[
\Per(F) \ge I(\haus^N(F)) \ge I(\haus^N(E)) - L|\haus^N(E)-\haus^N(F)| \ge \Per(E) - L\haus^N(F\Delta E).
\]

For the second part of the claim let us exploit \cite[Equation (3.51) in Remark 3.25]{AntonelliPasqualettoPozzetta21}. According to the latter, one can find $R'>0$ possibly smaller than $R$ and only depending on $K,N,v_0,V_1,V_2,V_3$ such that, calling $v(N,K/(N-1),r)$ the volume of the geodesic ball of radius $r$ in the model of constant sectional curvature $K/(N-1)$ and dimension $N$, one has, for some constant $\widetilde C_1$ only depending on $N,K,v_0$, that
\[
\Per(E,B_r(x))\leq \frac{1+\Lambda\widetilde C_1 v(N,K/(N-1),r)^{1/N}}{1-\Lambda\widetilde C_1 v(N,K/(N-1),r)^{1/N}}\Per(F,B_r(x)),
\]
for any $x\in X$, $r\in(0,R']$, and any $F$ such that $E\Delta F \Subset B_r(x)$. Hence, by taking $R'$ smaller if needed, and by using that there exists $\widetilde C_2$ only depending on $N,K$ such that $v(N,K/(N-1),r)\leq \widetilde C_2r^N$ for every $r\leq 1$, we get the conclusion in \eqref{eqn:OmegaMinimo}. 

The last part of the claim follows arguing as in the proof of \cite[Proposition 3.27]{AntonelliPasqualettoPozzetta21}.
\end{proof}

\subsection{Consequences}

From the previous results on the concavity of the isoperimetric profile one can prove that in the $N$-dimensional $\RCD(K,N)$ spaces, isoperimetric regions of sufficiently small volume are connected. If $K=0$, the conclusion holds for all volumes.
\begin{corollary}\label{cor:GHRnRegioniConnesse}
Let $(X,\dist,\haus^N)$ be an $\RCD(K,N)$ space with $N\geq 2$. Let us assume that $\inf_{x\in X}\haus^N(B_1(x))\geq  v_0>0$. Let $\varepsilon$ be such that the isoperimetric profile $I$ is strictly subadditive on $(0,\varepsilon)$. Such an $\varepsilon>0$ exists thanks to item (2) of \autoref{cor:EstimateDerivativeProfileAndSubadditivity}, and, if $K=0$, one can take $\varepsilon=\haus^N(X)$.\\
{ Let $E=E^{(1)}$ be an isoperimetric region in $X$ with $\haus^N(E)< \varepsilon$.} Then $E$ is connected. If in addition $\haus^N$ is finite, then $E$ is simple (i.e.\ $E$ and $X\setminus E$ are indecomposable) and $E^{(0)}$ is connected. 
\end{corollary}

\begin{proof}
{We recall that when we deal with an isoperimetric region $E$, we are always considering it is open by taking $E=E^{(1)}$.} From item (2) of \autoref{cor:EstimateDerivativeProfileAndSubadditivity} one gets that there exists $\varepsilon>0$ such that $I$ is strictly subadditive on $(0,\varepsilon)$. Notice that if $K=0$, we have that $I^{N/(N-1)}$ is concave as a consequence of item (2) of \autoref{thm:BavardPansu}, hence in this case $I$ is striclty subadditive on $(0,\haus^N(X))$.

Assume $\Omega$ is an isoperimetric region of volume $V<\varepsilon$. We prove first that $\Omega$ is indecomposable of volume $V$. Suppose by contradiction it is decomposable. Hence $\Omega=\Omega_1\cup\Omega_2$ with $\haus^N(\Omega_1)+\haus_N(\Omega_2)=V$ and $\Per(\Omega_1)+\Per(\Omega_2)=\Per(\Omega)=I(V)$. Hence 
$$
I(V)=\Per(\Omega_1)+\Per(\Omega_2)\geq I(\haus^N(\Omega_1))+I(\haus^N(\Omega_2))>I(V),
$$
where the last inequality is due to the fact that $I$ is strictly subadditive on $(0,\varepsilon)$. Hence we reach a contradiction.

To prove that $E$ is connected, we argue by contradiction: suppose there exist non-empty open sets $U,V\subseteq X$ such that $U\cap V=\emptyset$ and $U\cup V=E$.
Being $U$, $V$, $E$ open, we have that $\partial E=\partial U\cup\partial V$. Since we know that $\partial E=\partial^e E$ (recall that we are always assuming that
$E=E^{(1)}$), we deduce that $\haus^{N-1}(\partial U)<+\infty$ and accordingly $U$ is a set of finite perimeter, see e.g., \cite{Lahti20}.
Now consider the BV function $f\coloneqq\chi_U$. Again thanks to the fact that $U$ and $V$ are open, we get $|Df|(E)=0$. Given that $\RCD(K,N)$ spaces
have the two-sidedness property in the sense of \cite[Definition 1.28]{BonicattoPasqualettoRajala20}) (see \cite[Example 1.31]{BonicattoPasqualettoRajala20})
and $E$ is indecomposable, we deduce from \cite[Theorem 2.5]{BonicattoPasqualettoRajala20} that $f$ is $\haus^N$-a.e.\ constant on $E$, which leads to a
contradiction. Therefore, the set $E$ is connected.

Now assume that $\haus^N(X)<+\infty$. Then $X\setminus E$ is an isoperimetric set of its own volume, thus accordingly the last part of the statement follows
from the first one applied to $X\setminus E$.
\end{proof}


\begin{remark}
Classical examples show that, even for smooth compact Riemannian manifolds, connectedness of isoperimetric regions might fail for volumes bounded away from zero if the Ricci curvature is negative somewhere, see for instance \cite[Remark 2.3.12]{Bayle03}.
\end{remark}

From the strictly subadditivity of the isoperimetric profile (for small volumes, if $K<0$), we infer that there is at most one component in the asymptotic mass decomposition result in \autoref{thm:MassDecompositionINTRO} (for small volumes if $K<0$). This is understood in the following statement.

\begin{lemma}\label{lem:IsoperimetricAtFiniteOrInfinite}
Let $(X,\dist,\haus^N)$ be an $\RCD(K,N)$ space with $N\geq 2$. Let us assume that $\inf_{x\in X}\haus^N(B_1(x))\geq  v_0>0$. Let $\varepsilon$ be such that the isoperimetric profile $I$ is strictly subadditive on $(0,\varepsilon)$. Such an $\varepsilon>0$ exists thanks to item (2) of \autoref{cor:EstimateDerivativeProfileAndSubadditivity} and, if $K=0$, one can take $\varepsilon=\haus^N(X)$.
\begin{enumerate}
    \item Let $\{\Omega_i\}_{i\in\mathbb N}$ be a minimizing (for the perimeter) sequence of bounded finite perimeter sets of volume $v<\varepsilon$ in $X$. Then, if one applies \autoref{thm:MassDecompositionINTRO}, either $\overline N=0$, or $\overline N=1$ and $\haus^N(\Omega)=0$.
    
    \item Let $X_1,\ldots, X_{\overline{N}}$ be pmGH limits of $X$ along sequences of points $\{p_{i,j}\}_{i \in \setN}$, for $j=1,\ldots,\overline{N} \in \setN \cup\{+\infty\}$. Let $\Omega=E \cup \bigcup_{j=1}^{\overline{N}} E_j$, with $E\subset X, E_j \subset X_j$ be a set achieving the infimum in \eqref{eqn:GeneralizedIsoperimetricProfile} for some $v<\varepsilon$. Then exactly one component among $E,E_1,\ldots,E_{\overline{N}}$ is nonempty.
\end{enumerate}
In particular, for any $v<\varepsilon$ there is an $\RCD(K,N)$ space $(Y,\dist,\haus^N)$ which is either $X$ or a pmGH limit of $X$ along a sequence $\{p_i\}_i \subset X$, and a set $E \subset Y$ such that $\haus^N(E)=v$ and $I_X(v)=\Per(E)$.
\end{lemma}

\begin{proof}
Let us prove the first item, and to this aim we adopt the notation of \autoref{thm:MassDecompositionINTRO}. Let us assume that the assertion is not true. Hence there are $j\geq 2$ nonempty sets $E_1,\dots,E_j$ among $\Omega,Z_1,\dots,Z_{\overline{N}}$ such that, due to \eqref{eq:UguaglianzeIntro} and the fact that $E_1,\dots,E_j$ are isoperimetric in their own spaces,
$$
I(v)=\sum_{k=1}^j I_{X_k}(\haus^N(E_k)),
$$
and $\sum_{k=1}^j \haus^N(E_k)=v$. Using that, for every $v>0$ and every $k\in \{1,\dots,j\}$, we have $I_{X_k}(v)\geq I(v)$, see \cite{AntonelliNardulliPozzetta}, we thus conclude that 
$$
I(v)\geq \sum_{k=1}^j I(\haus^N(E_k)),
$$
which is in contradiction with the fact that $I$ is strictly subadditive on $(0,v)\subset(0,\varepsilon)$, and $j\geq 2$. 

The second item analogously follows from the strict subadditivity of the profile $I$ of $X$ and the identity in \autoref{prop:ProfileDecomposition}.
\end{proof}


As we already remarked, on a smooth Riemannian manifold $(M^N,g)$, the barrier $c$ obtained applying \autoref{thm:Isoperimetrici} to an isoperimetric region $E$ is unique and it coincides with the value of the (constant) mean curvature of the regular part of $\partial E$. Moreover, $t\mapsto \Per(E_t)$ is always differentiable at $t=0$ and 
\begin{equation}
\frac{\di }{\di t}|_{t=0}\Per(E_t)=c\Per(E)\, .
\end{equation}
A well known consequence of this observation is the fact that the isoperimetric profile is differentiable with derivative $I'(v)=c$ at any volume $v\in(0,\infty)$ such that there exists a unique isoperimetric region $E$ of volume $v$ (with constant mean curvature of the boundary equal to $c$). Below we partially generalize this statement to the present context.

\begin{corollary}\label{cor:MeanCurvatureDifferentialIsoperimetricProfile}
Let $(X,\dist,\haus^N)$ be an $\RCD(K,N)$ space such that $\haus^N(B_1(x))\geq v_0>0$ for any $x \in X$.\\
Let $v\in (0,\haus^N(X))$ and let $\Omega, Z_1,\ldots, Z_{\overline{N}}$ be isoperimetric sets given by \autoref{thm:MassDecompositionINTRO} applied to the minimization problem at volume $v$\footnote{We understand there is just one isoperimetric region $\Omega \subset X$ if $X$ is compact.}. Let $c$ be any barrier given by \autoref{thm:Isoperimetrici} applied on either $\Omega, Z_1, \ldots,$ or $Z_{\overline{N}}$. Then
\begin{equation}\label{eqn:MeanCurvatureInequality}
I_+'(v)\leq c \leq I_-'(v).
\end{equation}
Hence, if $I$ is differentiable at $v$, if $E \in \{\Omega, Z_1,\ldots, Z_{\overline{N}}\}$, then the barrier given by \autoref{thm:Isoperimetrici} applied to $E$ is unique and equal to $I'(v)$. In particular, if $I$ is differentiable at $v$ and $E$ is an isoperimetric region on $X$ for the volume $v>0$, we have that the barrier given by \autoref{thm:Isoperimetrici} applied to $E$ is unique and equal to $I'(v)$.
\end{corollary}

\begin{proof}
Let $v\in (0,\haus^N(X))$ be fixed. Take $\Omega_i$ a minimizing sequence of bounded sets of volume $v$ and let $\Omega,Z_j$ be the isoperimetric regions in $X,X_j$ respectively, according to the notation of \autoref{thm:MassDecompositionINTRO}.  
Let $c$ be any barrier as in the statement.

Let $E$ be an arbitrary isoperimetric region among $\Omega,Z_j$ in the spaces $X,X_j$. Let us set $\beta(t):=\haus^N(E_t)+\sum_{T\in\{\Omega,Z_1,\dots,Z_{\overline N}\}, T\neq E}\haus^N(T)$, where $E_t$ is the $t$-tubular neighbourhood of $E$ for $t\in (-\varepsilon,\varepsilon)$, with $\varepsilon>0$ small enough, see the discussion after \eqref{eq:use}.
{ Arguing as in \eqref{eqn:FINAL}, we reach the conclusion.}

%
%
%
\end{proof}

\begin{remark}
Letting $X$, $v\in (0,\haus^N(X))$ and $\Omega, Z_1,\ldots, Z_{\overline{N}}$ as in \autoref{cor:MeanCurvatureDifferentialIsoperimetricProfile}, if $I$ is differentiable at $v$, then the function $t\mapsto \Per(E_t)$ is differentiable at $t=0$ and its derivative is $c\Per(E)$, for any $E \in \{\Omega, Z_1,\ldots, Z_{\overline{N}}\}$ where $c$ is a barrier for $E$.
This follows by the argument in the proof of \autoref{cor:MeanCurvatureDifferentialIsoperimetricProfile}.

Moreover, if $I$ is differentiable at $v$, every isoperimetric set $\Omega,Z_1,\ldots,Z_{\overline{N}}$ has only one possible barrier $c=I'(v)$. 
\end{remark}

Another consequence of the sharp concavity properties of the isoperimetric profile are uniform diameter bounds for isoperimetric regions of small volume, in great generality, and any volume if the underlying space is $\RCD(0,N)$ with Euclidean volume growth.

\begin{proposition}\label{lem:BoundDiam}
For every $N\geq 2$ natural number, \(K\leq 0\), and \(v_0>0\), there exist constants \(\bar v=\bar v(K,N,v_0)>0\) and $C=C(K,N,v_0)>0$ such that the following holds.
Let $(X,\dist,\haus^N)$ be an $\RCD(K,N)$ space. Suppose that $\haus^N(B_1(x))\geq v_0$ holds for every $x\in X$. Let $E\subseteq X$ be an isoperimetric region. Then
\begin{equation}\label{eq:BoundDiamSmall}
\diam E\leq C\haus^N(E)^{\frac{1}{N}}\quad\text{ whenever }\haus^N(E)\leq\bar v\, .
\end{equation}
Moreover, for every \(N\geq 2\) natural number and \(A>0\), there exists a constant \(\tilde C=\tilde C(N,A)>0\) such that the following holds.
Let \((\X,\dist,\haus^N)\) be an \(\RCD(0,N)\) space satisfying
\[
{\rm AVR}(X,\dist,\haus^N)\coloneqq \lim_{r\to+\infty}\frac{\haus^N(B_r(\bar x))}{\omega_Nr^N}\geq A,
\]
where $\bar x\in X$ and $\omega_N$ is the Euclidean volume of the unit ball in $\mathbb R^N$. Then it holds that
\begin{equation}\label{eq:BoundDiamBig}
\diam E\leq\tilde C\haus^N(E)^{\frac{1}{N}}\quad\text{ for every isoperimetric region }E\subseteq X.
\end{equation}
\end{proposition}

\begin{proof}
By \autoref{lem:ConsequenceSmallDiameter}, \autoref{rmk:ConseqMJ}, and
item (2) of \autoref{prop:SharpenedConcavityAndLimit}, there exist constants $\vartheta=\vartheta(K,N,v_0)>0$, {\(\tilde v_1=\tilde v_1(K,N,v_0)>0\)}, and \(C_0=C_0(K,N,v_0)>0\) such that:
\begin{itemize}
\item[a)] \(I(v)\geq\vartheta v^{\frac{N-1}{N}}\) for every { \(v\leq \tilde  v_1\).}
\item[b)] \(I(v)/v^{\frac{N-1}{N}}\leq C_0\) for every { \(v\leq \tilde v_1\).}
\item[c)] { $I$ is concave on \([0,\tilde v_1]\)}.
\end{itemize}
Given any $x\in E$ and $r>0$, we define $m_x(r)\coloneqq\haus^N(E\cap B_r(x))$ and $E^x_r\coloneqq E\setminus B_r(x)$.
Set $v_E\coloneqq\haus^N(E)$ for brevity and suppose {\(v_E\leq \tilde v_1\). By c), the function \(v\mapsto I(v)/v\) is non-increasing on \([0,\tilde v_1]\),} so that
{
\[
\frac{I(v_E)}{v_E}\leq\frac{I\big(v_E-m_x(r)\big)}{v_E-m_x(r)},\quad\text{ for every }r>0.
\]
}
Multiplying both sides by \(v_E-m_x(r)\), we deduce that
{
\begin{equation}\label{eq:BoundDiam1}\begin{split}
\Per(E)&=I(v_E)\leq I\big(v_E-m_x(r)\big)+\frac{m_x(r)}{v_E}I(v_E)\leq \Per(E_r^x)+\frac{m_x(r)}{v_E}I(v_E).
\end{split}
\end{equation}}
By using \cite[Lemma 4.5]{Ambrosio02} we obtain $\Per(E^x_r)+\Per(E\cap B_r(x))\leq \Per(E)+2m'_x(r)$ for a.e.\ $r>0$.
Hence a) implies
\begin{equation}\label{eq:BoundDiam2}
\Per(E^x_r)\leq \Per(E)-\Per(E\cap B_r(x))+2m'_x(r)\leq \Per(E)-\vartheta m_x(r)^{\frac{N-1}{N}}+2m'_x(r)
\end{equation}
for a.e.\ $r>0$. Combining \eqref{eq:BoundDiam1} with \eqref{eq:BoundDiam2}, we thus obtain for a.e.\ \(r>0\) that
\begin{equation}\label{eq:BoundDiam3}
\vartheta m_x(r)^{\frac{N-1}{N}}-\frac{I(v_E)}{v_E}m_x(r)\leq 2m'_x(r).
\end{equation}
{ Once \eqref{eq:BoundDiam3} is obtained, one can argue as in \cite[Lemma 5.7]{LeonardiRitore} to deduce the claim \eqref{eq:BoundDiamSmall} with $\bar v = \tilde v_1$. We just sketch the main steps.}
Let
\begin{equation}\label{eq:zzDefrzero}
    {
r_0\coloneqq\frac{1}{\omega_N^{1/N}}\frac{\vartheta v_E}{4 I(v_E)} \overset{{\rm b)}}{\ge} \frac{\vartheta}{4 \omega_N^{1/N} C_0} v_E^{1/N}.
}
\end{equation}
{Defining $f(r) \eqdef 2e^{\frac{I(v_E)}{2v_E}r} m_x(r)$, \eqref{eq:BoundDiam3} implies that $f'(r) \ge 
2^{\frac1N-1}\vartheta e^{\frac{I(v_E)}{2v_E \, N}r} f^{\frac{N-1}{N}}$
for a.e. $r \in (0,r_0)$. Hence \eqref{eq:zzDefrzero} implies
\begin{equation}\label{eqn:Stimasottooo}
    \begin{split}
        2^{\frac1N}e^{\frac{\vartheta}{8N \omega_N^{1/N}}} m_x(r_0)^{\frac1N}
        &= f(r_0)^{\frac1N} \ge \int_{r_0/2}^{r_0} (f(r)^{\frac1N})' \de r  \ge
        c(K,N,v_0) r_0,
    \end{split}
\end{equation}
where $c(K,N,v_0)>0$ may change from line to line. Hence \eqref{eqn:Stimasottooo} implies
\begin{equation}\label{eq:BoundDiam5}
    m_x(r_0) \ge c(K,N,v_0) r_0^N.
\end{equation}
Considering a maximal family $\mathcal F$ of pairwise disjoint open balls $B$ having radius $r_0$ and center $c(B)$ in $E$, since
$\{B_{2r_0}(c(B))\,:\,B\in\mathcal F\}$ is a covering of $E$, then $\#\mathcal F \le c(K,N,v_0) v_E/ r_0^N$.
The set
$U\coloneqq\bigcup_{B\in\mathcal F}B_{2r_0}(c(B))$ contains $E$, which is connected by \autoref{cor:GHRnRegioniConnesse},
hence $U$ must be connected as well. Then
\begin{equation}\label{eq:BoundDiam6}
\begin{split}
    \diam E&\leq\diam U\leq\sum_{B\in\mathcal F}\diam B_{2r_0}(c(B))\leq 4r_0\#\mathcal F
 {\overset{\eqref{eq:zzDefrzero}}{\le} C v_E^{\frac{1}{N}},}
\end{split}
\end{equation}
{ for a suitable constant $C=C(K,N,v_0)$, proving \eqref{eq:BoundDiamSmall}.}
}

%
%

{
The second part of the statement is a consequence of the first part, and a scaling argument. Indeed, from the first part one gets that the following holds for every $K\leq 0, R> 0$, $N\in \mathbb N\cap [2,\infty)$, and $v_0>0$. For every $\RCD(KR^{-2}
,N)$ space $(X,\dist,\haus^N)$
for which $\inf_{x\in X}\haus^N (B_R(x)) \geq v_0R^N$, then every isoperimetric region $E$ with $\haus^N(E)\leq \bar v (K,N,v_0)R^N$ has diameter bounded from above by $C(K
, N, v_0)\haus^N(E)^{\frac{1}{N}}$. 

If now $(X,\dist,\haus^N)$ is an $\RCD(0,N)$ space with $\mathrm{AVR}(X,\dist,\haus^N)\geq A>0$, then one has $\inf_{x\in X}\haus^N(B_R(x))\geq A\omega_NR^N$ for every $R>0$. Thus, taking $E$ an isoperimetric region and using what we said above with $K=0$ and $R$ sufficiently large, we get
\[
\diam(E)\leq C(0,N,A\omega_N)\haus^N(E)^{\frac{1}{N}}.
\]
}
\end{proof}

\begin{remark}
\autoref{lem:BoundDiam} is a three-fold generalization of \cite[Lemma 4.9]{NardulliOsorio20}. Indeed, we prove the statement in the non-smooth $\RCD(K,N)$ case with reference measure $\haus^N$ and $\haus^N(B_1(x))\geq v_0>0$ for every $x\in X$, while the authors in \cite{NardulliOsorio20} deal with smooth Riemannian manifolds. Moreover, in the smooth case, our setting corresponds to the \textit{weak bounded geometry} hypothesis in \cite{NardulliOsorio20}, while \cite[Lemma 4.9]{NardulliOsorio20} is proved under the stronger \textit{mild bounded geometry} hypothesis. Finally, we are able to drop the hypothesis (H) in \cite[Lemma 4.9]{NardulliOsorio20}. Hence, \autoref{lem:BoundDiam} is new even in the smooth setting and sharpens the previous \cite[Lemma 4.9]{NardulliOsorio20}.

Notice that \autoref{lem:BoundDiam} generalizes also \cite[Lemma 5.5]{LeonardiRitore} to the setting of $\RCD(0,N)$ spaces with a uniform bound below on the volume of unit balls. 
\end{remark}

{

We conclude by stating a stability result for sequences of isoperimetric sets $E_i$ converging in $L^1$ to a limit set, where the $L^1$ convergence improves to Hausdorff convergence. In \autoref{thm:ConvergenceBarriersStability} below, observe that no uniform hypotheses on the mean curvature barriers for the $E_i$'s are assumed. Instead a uniform bound on such barriers follows from the fine properties we proved on the isoperimetric profile. We mention that analogous stability results for mean curvature barriers have been independently considered in the recent \cite{Ketterer21}.

\begin{theorem}\label{thm:ConvergenceBarriersStability}
Let $(X_i,\dist_i,\haus^N,x_i)$ be a sequence of $\RCD(K,N)$ spaces converging to $(Y,\dist_Y,\haus^N,y)$ in pmGH sense, and let $(Z,\dist_Z)$ be a space realizing the convergence. Assume that $\haus^N(B_1(p))\ge v_0>0$ for any $p \in X_i$ and any $i$. Let $E_i \subset X_i$, $F\subset Y$.

If $E_i$ is isoperimetric, $E_i\subset B_R(x_i)$ for some $R>0$ for any $i$, $c_i$ is a mean curvature barrier for $E_i$ for any $i$, $E_i\to F $ in $L^1$-strong, and $0<\lim_i \haus^N(E_i)< \lim_i \haus^N(X_i)$, then
\begin{equation}\label{eq:ConvergenceBarriersStability}
    \begin{split}
        \text{$F$ is isoperimetric}\, ,& \\
         |c_i| \le L& \qquad \text{for any $i$ large enough}\, ,\\
        |D\chi_{E_i}|\to |D\chi_F|  & \qquad\text{in duality with $C_{\rm bs}(Z)$}\, ,\\
        \partial E_i \to \partial F\, , \,\, 
        \overline{E}_i \to \overline{F}  & \qquad\text{in Hausdorff distance in $Z$}\, ,
    \end{split}
    \end{equation}
where $L=L(K,N,v_0,{\inf_i \haus^N(E_i), \sup_i \haus^N(E_i)})>0$. In particular, the mean curvature barriers $c_i$ converge up to subsequence to a mean curvature barrier for $F$.
\end{theorem}

The proof of \autoref{thm:ConvergenceBarriersStability} follows by well-established arguments. In fact, by a classical contradiction argument as in \cite{AFP21, AntonelliNardulliPozzetta}, it follows that $\Per(E_i)\to \Per(F)$ and $F$ is isoperimetric. In particular, $|D\chi_{E_i}|\to |D\chi_F|$ in duality with $C_{\rm bs}(Z)$.

By \autoref{cor:UniformLipschitzProfile}, \autoref{cor:UniformRegularityIsop}, {and \autoref{cor:MeanCurvatureDifferentialIsoperimetricProfile}},   the mean curvature barriers $c_i$ are uniformly bounded and the $E_i$ satisfy uniform density estimates at boundary points, independently of $i$. This readily implies that any converging sequence $q_i \in \partial E_i$ must converge to a point in $\partial F$ (compare with \cite[Theorem 2.43]{MoS21}), and then Kuratowski convergence of $\partial E_i$ to $\partial F$ is achieved, which easily implies \eqref{eq:ConvergenceBarriersStability}.

Finally, if $c_i\to c\in\R$ up to subsequence, integrating a sequence of strongly $H^{1,2}$-converging Lipschitz functions along $(X_i,\dist_i,\haus^N,x_i) \to (Y,\dist_Y,\haus^N,y)$ with respect to the Laplacian of the signed distance functions from $E_i$, one gets that the signed distance from $F$ satisfies the adimensional bounds \eqref{eq:adimensional laplacian comparison}. Arguing as in Step 2 of the proof of \autoref{thm:Isoperimetrici}, the inequalities are readily improved to get that $c$ is a mean curvature barrier for $F$.

}

\printbibliography[title={References}]

\typeout{get arXiv to do 4 passes: Label(s) may have changed. Rerun} 

\end{document}